\documentclass[11pt]{amsart}

\usepackage{amssymb}

\usepackage[margin=1in]{geometry}

\usepackage[bookmarksopen, bookmarksnumbered]{hyperref}

\usepackage{comment}

\usepackage{pgfplots}

\newtheorem{thm}{Theorem}[section]
\newtheorem{prop}[thm]{Proposition}
\newtheorem{lem}[thm]{Lemma}
\newtheorem{cor}[thm]{Corollary}

\newtheorem{conj}[thm]{Conjecture}

\theoremstyle{definition}
\newtheorem{definition}[thm]{Definition}

\theoremstyle{remark}

\theoremstyle{definition}

\newtheorem{case}{Case}

\makeatletter
\@addtoreset{case}{proof}
\makeatother

\numberwithin{equation}{section}

\DeclareSymbolFont{AMSb}{U}{msb}{m}{n}
\DeclareMathSymbol{\N}{\mathbin}{AMSb}{"4E}
\DeclareMathSymbol{\Z}{\mathbin}{AMSb}{"5A}
\DeclareMathSymbol{\R}{\mathbin}{AMSb}{"52}
\DeclareMathSymbol{\Q}{\mathbin}{AMSb}{"51}
\DeclareMathSymbol{\I}{\mathbin}{AMSb}{"49}
\DeclareMathSymbol{\C}{\mathbin}{AMSb}{"43}
\DeclareMathSymbol{\F}{\mathbin}{AMSb}{"46}

\DeclareMathOperator{\Vol}{Vol}
\newcommand*\diff{\mathop{}\!\mathrm{d}}

\newcommand{\im}{\mathbf{i}}
\newcommand{\me}{\mathrm{e}}

\newenvironment{sproof}{%
  \proof}{\endproof}

\makeatletter
\def\@tocline#1#2#3#4#5#6#7{\relax
  \ifnum #1>\c@tocdepth 
  \else
    \par \addpenalty\@secpenalty\addvspace{#2}%
    \begingroup \hyphenpenalty\@M
    \@ifempty{#4}{%
      \@tempdima\csname r@tocindent\number#1\endcsname\relax
    }{%
      \@tempdima#4\relax
    }%
    \parindent\z@ \leftskip#3\relax \advance\leftskip\@tempdima\relax
    \rightskip\@pnumwidth plus4em \parfillskip-\@pnumwidth
    #5\leavevmode\hskip-\@tempdima
      \ifcase #1
       \or\or \hskip 1em \or \hskip 2em \else \hskip 3em \fi%
      #6\nobreak\relax
    \dotfill\hbox to\@pnumwidth{\@tocpagenum{#7}}\par
    \nobreak
    \endgroup
  \fi}
\makeatother

\begin{document}

\title{Refinements of the 2-dimensional Strichartz Estimate on the Maximum Wave Packet}

\author{Hong Wang}
\address{Department of Mathematics\\
Massachusetts Institute of Technology\\
Cambridge, MA 02139}
\email[H.~Wang]{hongwang@mit.edu}

\author{Lingfu Zhang}
\address{Department of Mathematics\\
Massachusetts Institute of Technology\\
Cambridge, MA 02139}
\email[L.~Zhang]{lfzhang@mit.edu}

\maketitle


\begin{abstract}
The Strichartz estimates for Schr\"{o}dinger equations can be improved when the data is spread out in either physical or frequency space. In this paper we give refinements of the 2-dimensional homogeneous Strichartz estimate on the maximum size of a single wave packet. Different approaches are used in the proofs, including arithmetic approaches, polynomial partitioning, and the $l^2$ Decoupling Theorem, for different cases. We also give examples to show that the refinements we obtain cannot be further improved when $2 \leq p \leq 4$ and $p = 6$.
\end{abstract}


\begin{center}
\begin{minipage}[t]{0.85\linewidth}
\tableofcontents
\end{minipage}
\end{center}


\bigskip
\section{Introduction}

We start with the \emph{initial value problem} (IVP) of the $n + 1$ dimensional \emph{linear Sch\"ordinger equation}:
\begin{equation}  \label{eq:lse}
\begin{cases}
\im u_t - \Delta u = 0, \\
u(x, 0) = u_0(x)
\end{cases}
\end{equation}
where $u : \R^n \times \R \rightarrow \C$ is the unknown, $u_0 : \R^n \rightarrow \C$ is the given initial data, and $\Delta = \sum_{i=1}^n \frac{\partial^2}{\partial x_i^2}$ is the Laplacian operator.
Assuming that $u_0$ is Schwartz, 
denote $f = \widehat{u}$ to be the Fourier transform. As a classical result, 
the solution to (\ref{eq:lse}) is given by
\begin{equation}
u(x, t) = \int_{\R^n} \me^{\im(|\omega|^2 t + \omega \cdot x)} f(\omega) \diff \omega  ,
\end{equation}
see, e.g. \cite[Section 4.3]{book:857896} \cite[Section 2.2]{tao2006local}, in slightly different notations.
The well-known Strichartz estimate bounds the ($n+1$ dimensional) $L^p$ norm of the solution $u$ by the ($n$-dimensional) $L^2$ norm of $f$.
\begin{thm} \cite[Theorem 1]{strichartz1977restrictions} \label{thm:str}
For $p = \frac{2(n+2)}{n}$, there is
\begin{equation}  \label{eq:ep1}
\| u \|_{L^p(\R^{n}\times \R)} \lesssim \| f \|_{L^2(\R^n)}  ,
\end{equation}
and the constant behind $\lesssim$ only relies on $p$ and $n$.
\end{thm}
This result can also be derived by the Tomas-Stein theorem in the restriction problem (see e.g., \cite{stein1986oscillatory} \cite{tomas1975restriction}, and see \cite{tao2004some} for a survey of the restriction problem).
Let $E: L^2(\R^n) \rightarrow L^p(\R^{n+1})$ denote the operator sending $f$ to $u$, and it is also referred to as the extension operator in some literature.

In this text we focus on dimension $n = 1$ (thus $2$ time-space dimensions), where $p=6$ in Theorem \ref{thm:str}.
As Fourier transform is unitary, and the space $L^2$ norm of the solution to (\ref{eq:lse}) is conserved, there is
\begin{equation}   \label{eq:ep2}
\| Ef \|_{L^2(B_R)} \leq \| Ef \|_{L_x^2L_t^2([-R, R])} = (2R)^{\frac{1}{2}} \| f \|_{L^2} ,
\end{equation}
where $B_R$ is the ball centered at the origin with radius $R$.
Applying H\"older's inequality to (\ref{eq:ep1}) and (\ref{eq:ep2}), there is
\begin{equation}  \label{eq:trf}
\| Ef \|_{L^p(B_R)} \lesssim R^{\frac{3}{2p} - \frac{1}{4}} \| f \|_{L^2}  ,
\end{equation}
for any $2\leq p \leq 6$ and $R > 0$.

The estimate (\ref{eq:trf}) is sharp for all $2 \leq p \leq 6$, but if we make additional assumptions about the function $f$, then we can prove better estimates.
Specifically, we break $f$ into wave packets
\begin{equation}
f = \sum_{\theta, v} f_{\theta, v}  ,
\end{equation}
where each $f_{\theta, v}$ is supported in an interval of length $O(R^{-\frac{1}{2}})$, and its Fourier transform is essentially supported in an interval of length $O(R^{\frac{1}{2}})$.
In $\R^2$ each $Ef_{\theta, v}$ is essentially supported in a $R^{\frac{1}{2}} \times R$ tube inside $B_R$.
We denote
\begin{equation}
S = \frac{\max_{\theta, v} \|f_{\theta, v}\|_{L^2}   }{\|f\|_{L^2}} ,
\end{equation}
which measures the size of a single wave packet compared with the total size.
Then key results of this paper, which are special cases for Theorem \ref{thm:main}, are as following. 
\begin{thm}[Key Points]  \label{thm:key}
For arbitrarily small $\epsilon > 0$, there is
\begin{equation}  \label{eq:key4}
\| Ef \|_{L^4(B_R)} \lesssim_{\epsilon} R^{\frac{1}{8} + \epsilon} S^{\frac{1}{4} -\epsilon} \| f \|_{L^2}  ,
\end{equation}
\begin{equation}  \label{eq:key5}
\| Ef \|_{L^5(B_R)} \lesssim_{\epsilon} R^{\frac{1}{20} + \epsilon} S^{\frac{1}{5} -\epsilon} \| f \|_{L^2}  ,
\end{equation}
\begin{equation}  \label{eq:key6}
\| Ef \|_{L^6(B_R)} \lesssim_{\epsilon} R^{\frac{1}{6} + \epsilon} S^{\frac{2}{3} -\epsilon} \| f \|_{L^2}  .
\end{equation}
\end{thm}
Compare with (\ref{eq:trf}): for $p = 4, 5$, we have an extra power of $S$ in the right hand side of (\ref{eq:key4}) and (\ref{eq:key5}), respectively;
for $p = 6$, (\ref{eq:key6}) is a refinement only when $S \lesssim R^{-\frac{1}{4}}$.

Note that essentially there is $R^{-\frac{1}{2}} \lesssim S \lesssim 1$.
Using H\"older's inequality, we can obtain refinements of (\ref{eq:trf}) for all $2 < p \leq 6$.
We also give examples, which imply that our results for $2 < p \leq 4$ cannot be further improved, and for $5 \leq p \leq 6$ they can not be improved either if we assume that $S \lesssim R^{-\frac{1}{4}}$.

The motivation of these refinements is the observation that (\ref{eq:trf}) is not sharp when $f$ is spread out in either physical or frequency space.
To be more precise, let's look at two examples.
First consider 
$f_0: \R \rightarrow \C$, which is a smooth bump function supported in $[-2, 2]$, and equals $1$ in $[-1, 1]$.
Through a typical argument of stationary phase, we can bound $|Ef_0(x, t)|$ by $|t|^{-\frac{1}{2}}$ when $|t| \gtrsim |x|$, and $|x|^{-1}$ when $|x| \lesssim |t|$.
Then there is
\begin{equation}
\| Ef_0 \|_{L^p(B_R)} \lesssim 
\begin{cases}
R^{\frac{2}{p} - \frac{1}{2}} \sim R^{\frac{2}{p} - \frac{1}{2}} \| f_0 \|_{L^2}, & 2 \leq p \leq 4 \\
1 \sim \| f_0 \|_{L^2}, & p > 4 ,
\end{cases}
\end{equation}
which is strictly better than (\ref{eq:trf}) for $2 < p < 6$.
We also have a sharp example for (\ref{eq:trf}): let 
$f_1$ be a smooth bump function supported in $\left[ - 2R^{-\frac{1}{2}}, 2R^{-\frac{1}{2}} \right]$, and equals $1$ in $\left[ - R^{-\frac{1}{2}}, R^{-\frac{1}{2}} \right]$.
When $|x| < R^{\frac{1}{2}}$ and $|t| < R$ there is $|Ef_1(x, t)| \sim 1$.
This implies that
\begin{equation}
\| Ef_1 \|_{L^p(B_R)} \gtrsim 
R^{\frac{3}{2p} - \frac{1}{2}} \sim R^{\frac{3}{2p} - \frac{1}{4}} \| f_1 \|_{L^2},\quad 2 \leq p \leq 6 ,
\end{equation}
which saturates the estimate (\ref{eq:trf}) for all $2 \leq p \leq 6$.

When comparing these two examples, we note that in a wave packet decomposition (of size $R$), $f_0$ has $\sim R^{\frac{1}{2}}$ nonzero wave packets of (almost) equal size, while $f_1$ has $\sim 1$ nonzero wave packet.
Then for $f_0$, there is $S \sim R^{-\frac{1}{4}}$; and for $f_1$, there is $S \sim 1$.
This motivates us to refine (\ref{eq:trf}) when $S$ is small, by inserting a power of $S$ to its right hand side.




For the proofs, we use different approaches for different cases of this problem.
For $p=4$, which is an even integer, we use arithmetic approaches to fully exploit the structure of the wave packet decomposition.
For $p=5$ we adapt polynomial partitioning, a method from incidental geometry and introduced by Larry Guth to the restriction estimates \cite{guth2016restriction} \cite{1603.04250}.
For $p=6$, we apply the recent proved $l^2$ Decoupling Theorem \cite{MR3374964}.
Let us point out that the methods of polynomial partitioning and $l^2$ Decoupling Theorem capture different aspects in the refinements:
the polynomial method deals with the intersecting of ``tubes'' in $B_R$, while $l^2$ Decoupling Theorem deals with cancellations in frequencies.

At the end we make a conjecture about the sharp estimates, for the remaining cases.
We expect that a combination of the approaches used in this text will eventually settle this problem.

\bigskip

\noindent\emph{Acknowledgment}
This work was conducted at the 2016 Summer Program in Undergraduate Research (SPUR) of the MIT Department of Mathematics, where the first author was the graduate mentor.
We would like to thank Prof. Larry Guth for proposing this problem,
and Prof. David Jerison and Prof. Ankur Moitra for  many  useful  discussions
and for suggesting the direction of this work.


\bigskip
\section{Problem Setup}  \label{sec:bg:wpd}

\subsection{Notations and Standard Results}

In this section we introduce notations and some some basic results under this context. 
These results will also be widely referred to in the proofs in subsequent sections.

There is a standard result from the method of stationary phase, and we state it here for future reference.
The proof of this can be easily found in most harmonic analysis textbooks
(see, e.g. \cite[Chapter VIII]{mit.00066085419930101} \cite[Chapter 6]{book:430974}).

\begin{thm}[Principle of non-stationary phase] \label{thm:sp}
Let both $\phi : \R^n \rightarrow \R$ and
$\varphi : \R^n \rightarrow \C$ be smooth,
and $|\nabla \phi|$ is bounded away from zero in the support of $\varphi$.
Then for any positive integer $N$,
we have
\begin{equation}
\left| \int_{\R^n} e^{\im \lambda \phi(x)} \varphi(x) \diff x  \right| \lesssim_{N} \lambda^{-N}  .
\end{equation}
\end{thm}
In other words,
we say that the integral fast decays as $\lambda$ grows.

We formally give the definition of the extension operator $E$.
By rescaling it suffices to consider the case where $f$ is supported in $[-1, 1]$.
\begin{definition}
For any smooth function $f:[-1, 1]\rightarrow \C$, let
\begin{equation}
Ef(x, t) := \int_{[-1, 1]} e^{\im (\omega^2 t + \omega x)} f(\omega) \diff \omega  .
\end{equation}
\end{definition}

We also identify $E$ with the inverse Fourier transform in $\R^2$, in the following sense:
\begin{definition}
Denote $P$ to be the \emph{truncated parabola} in $\R^2$: $P:= \{ (\omega, \omega^2) : \omega \in [-1, 1] \}$.
Let $\diff \sigma$ be the distribution supported on $P$, such that its projection to the first coordinate is the uniform measure on $[-1, 1]$.
For each function $f : [-1, 1] \rightarrow \C$, denote $f^*: P \rightarrow \C$ to be the function with
$f^*(\omega, \omega^2) = f(\omega)$.
By direct computation, the $\R^2$ Fourier transform of $Ef$ is $\widehat{Ef} = f^* \diff \sigma$.
\end{definition}

We will use a function to ``cut out'' the $L^p$ norm inside $B_R$.
\begin{definition}
Let $\eta : \R^2 \rightarrow \C$ be a Schwartz function, such that 
$\widehat{\eta}$ is supported in $B_1$, nonnegative, and is constant $1$ inside $B_{0.99}$.
By direct computation, $\eta \sim 1$ inside $B_1$;
and by Theorem \ref{thm:sp},
\begin{equation}
|\eta(x, t)| \lesssim_{N} \left( 1 + |x| + |t| \right)^{-N}  ,
\end{equation}
for any positive integer $N$.

Let $\eta_R : \R^2 \rightarrow \C$ such that $\eta_R(x, t) = \eta(x / R, t/R)$.
Then $\eta_R \sim 1$ inside $B_R$ and
\begin{equation}   \label{eq:sf:bd}
|\eta_R(x, t)| \lesssim_{N} \left( 1 + (|x| + |t|)R^{-1} \right)^{-N}  ,
\end{equation}
for any positive integer $N$.
\end{definition}

\begin{definition}
For any measurable set $A \subset \R^2$, use $\Vol(A)$ to denote the Lebesgue measure (\emph{area}) of $A$.
\end{definition}

\subsection{Wave Packet Decomposition}

Now we state the wave packet decomposition that will be used in this text.
The version we adapt here is from \cite{tao2003sharp}, and
the sketch proof given below also follows the proof there.
We point out that there are some other versions of the wave packet decomposition in literature, and interested readers can find some discussions in \cite{bourgain42some} and \cite{tao1998bilinear}.

First let us introduce the notation of a \emph{tube}.
\begin{definition}
For any $R > 0$, $\theta \in [-1, 1]$
and $v \in\R$,
denote $RT_{\theta, v}$ to be the \emph{tube}:
\begin{equation}
RT_{\theta, v} = \left\{ (t, x) \left|  |x - v - \theta t| \leq R^{\frac{1}{2}}   \right. \right\}  .
\end{equation}
\end{definition}

For any smooth function $f:[-1, 1]\rightarrow \C$,
we can decompose it into wave packets,
each supported in an interval of length $\sim R^{-\frac{1}{2}}$,
and the corresponding $Ef$
is decomposed 
such that inside $B_R$, 
each wave packet is essentially supported in a tube.

More precisely, we have the following \emph{wave-packet decomposition of size $R$}.
\begin{prop}\protect{\cite[Lemma 4.1]{tao2003sharp}}   \label{prop:wpd}
For any $R > 0$ and smooth $f : [-1, 1] \rightarrow \C$,
there exists a decomposition 
\begin{equation} 
f = \sum_{\substack{\theta \in R^{-\frac{1}{2}}\Z \bigcap [-1, 1], \\ v \in R^{\frac{1}{2}}\Z }} f_{\theta, v}  = \sum_{\substack{\theta \in R^{-\frac{1}{2}}\Z \bigcap [-1, 1], \\ v \in R^{\frac{1}{2}}\Z }} c_{\theta, v} \phi_{\theta, v},
\end{equation}
where each $f_{\theta, v}$ (and $\phi_{\theta, v}$) is supported in $[\theta - 3R^{-\frac{1}{2}}, \theta + 3R^{-\frac{1}{2}}]$.
Each coefficient $c_{\theta, v} \in \C$ satisfies that
\begin{equation}  \label{eq:wpd:l2}
\sum_{\theta, v} |c_{\theta, v}|^2 \sim \|f\|_{L^2}^2  ,
\end{equation}
and for each $\theta_o \in R^{-\frac{1}{2}}\Z \bigcap [-1, 1], v_o\in R^{\frac{1}{2}}\Z$,
\begin{equation}  \label{eq:wpd:sl2}
| c_{\theta_o, v_o} | \lesssim \max_{\theta, v} \| f_{\theta, v} \|_{L^2} .
\end{equation}
Each $\phi_{\theta, v}$ also satisfies
\begin{equation}  \label{eq:wpd:bd}
|E\phi_{\theta, v}(x, t) | \lesssim_N R^{-\frac{1}{4}} \left( 1 + |x - v - \theta t|R^{-\frac{1}{2}} \right)^{-N}
\end{equation}
for any positive integer $N$.
\end{prop}

\begin{sproof}
First write $f = \sum_{\theta \in R^{-\frac{1}{2}}\Z \bigcap [-1, 1] } f_{\theta} $,
where each $f_{\theta}$ is supported in the $R^{-\frac{1}{2}}$-neighborhood of $\theta$.
Since the support for each $f_{\theta}$ overlaps with at most $2$ others, there is
\begin{equation}
\sum_{\theta} \| f_{\theta} \|_{L^2}^2 \sim \| f \|_{L^2}^2  .
\end{equation}

By the Poisson summation formula we find a Schwartz function $\gamma: \R \rightarrow \C$,
such that $\widehat{\gamma}$ is supported in $[-1, 1]$,
and $\sum_{k\in \Z} \gamma(x - k) = 1$ for any $x \in \R$.
Now denote
\begin{equation}
f_{\theta, v} = \left( \widehat{f}_{\theta} \cdot \gamma\left( R^{-\frac{1}{2}} (x - v) \right) \right)^{\vee }  ,
\end{equation}
then each $f_{\theta, v} = f_{\theta} \ast \gamma\left( R^{-\frac{1}{2}} (x - v) \right)^{\vee } $ is supported in the $3R^{-\frac{1}{2}}$-neighborhood of $\theta$.
Take
\begin{equation}
c_{\theta,v} : = R^{\frac{1}{4}} MEf_{\theta}(v, 0) ,
\end{equation}
where
\begin{equation}
MEf_{\theta}(x, 0) : = \sup_{r > 0} \frac{1}{2r} \int_{[x-r, x+r]} |Ef_{\theta}(x', 0)| \diff x'
\end{equation}
is the Hardy-Littlewood maximal function.

Take Schwartz function $\varphi : \R \rightarrow \C$,
such that $\widehat{\varphi} = 1$ inside the $3R^{-\frac{1}{2}}$-neighborhood of $\theta$
and vanishes outside the $4R^{-\frac{1}{2}}$-neighborhood of $\theta$.
And $\varphi$ also satisfies
\begin{equation}
|\varphi(x)| \lesssim_{N} R^{-\frac{1}{2}}\left(1 + |x|R^{-\frac{1}{2}} \right)^{-N}
\end{equation}
for any positive integer $N$.

Since $f_{\theta}$ is supported in the $3R^{-\frac{1}{2}}$-neighborhood of $\theta$,
we have that
$Ef_{\theta}(\cdot, 0) = Ef_{\theta}(\cdot, 0) \ast \varphi $.
This implies that
\begin{equation}
MEf_{\theta}(x_1, 0) \sim MEf_{\theta}(x_2, 0)  ,
\end{equation}
for any $|x_1 - x_2| \leq R^{\frac{1}{2}}$.
Then we conclude that
\begin{equation}  \label{eq:wpd:pf2}
\sum_{v} |c_{\theta, v}|^2 \sim \sum_{v} \int_{[v - R^{\frac{1}{2}}, v + R^{\frac{1}{2}}]} |MEf_{\theta}(x, 0)|^2 \diff x 
\lesssim \int |Ef_{\theta}(x, 0)|^2 \diff x \leq  \| f_{\theta} \|_{L^2}^2  .
\end{equation}

Let's consider (\ref{eq:wpd:sl2}).
For any $\theta_o$ and $v_o$,
if for each $x \in \left[v_o - R^{\frac{1}{2}}, v_o + R^{\frac{1}{2}}\right]$,
there is a neighborhood $\mathcal{A}_x$ of $x$,
such that $|\mathcal{A}_x| \leq R^{\frac{1}{2}}$, and
\begin{equation}  \label{eq:wpd:pf3}
MEf_{\theta_o}(x, 0) = \frac{1}{|\mathcal{A}_x|} \int_{\mathcal{A}_x} |Ef_{\theta_o}(x', 0)|\diff x' ,
\end{equation}
then $MEf_{\theta_o} = M\left( \chi_{o} Ef_{\theta_o}  \right)$ inside
$\left[v_o - R^{\frac{1}{2}}, v_o + R^{\frac{1}{2}}\right]$,
where $\chi_o$ is the indicator function of the interval
$\left[v_o - 2R^{\frac{1}{2}}, v_o + 2R^{\frac{1}{2}}\right]$,
and
\begin{multline}
|c_{\theta_o, v_o}|^2 \sim
\int_{\left[v_o - R^{\frac{1}{2}}, v_o + R^{\frac{1}{2}}\right]} |M\left( \chi_{o} Ef_{\theta_o} \right)|^2 \diff x 
\lesssim \int_{\left[v_o - 2R^{\frac{1}{2}}, v_o + 2R^{\frac{1}{2}}\right]} \left|Ef_{\theta_o}(x, 0)\right|^2 \diff x 
\\
\lesssim \int \sum_{k=-2}^2 \left|Ef_{\theta_o, v_o + kR^{\frac{1}{2}}}(x, 0)\right|^2 \diff x 
\lesssim \max_{\theta, v} \| f_{\theta, v} \|_{L^2}^2 .
\end{multline}
Otherwise, if (\ref{eq:wpd:sl2}) does not hold for each $x \in \left[v_o - R^{\frac{1}{2}}, v_o + R^{\frac{1}{2}}\right]$, meaning that there is $x_o \in [v_o - R^{\frac{1}{2}}, v_o + R^{\frac{1}{2}}]$,
and
a neighborhood $\mathcal{A}_o$ of $x_o$,
such that $|\mathcal{A}_o| > R^{\frac{1}{2}}$, and
\begin{equation}
MEf_{\theta_o}(x_o, 0) = \frac{1}{|\mathcal{A}_o|} \int_{\mathcal{A}_o} |Ef_{\theta_o}(x', 0)|\diff x' ,
\end{equation}
then there is a $v'$, covered by $\mathcal{A}_o$,
such that
\begin{equation}
|c_{\theta_o, v_o}|^2 \lesssim \int_{\left[v' - R^{\frac{1}{2}}, v' + R^{\frac{1}{2}}\right]} |Ef_{\theta_o}(x', 0)|^2 \diff x' \lesssim \| f_{\theta_o, v'} \|_{L^2}^2 \lesssim \max_{\theta, v} \| f_{\theta, v} \|_{L^2}^2 .
\end{equation}

Now we show (\ref{eq:wpd:bd}).
It suffices to prove the following estimate:
\begin{equation}
\left| Ef_{\theta, v}(x, t) \right| \lesssim_N R^{-\frac{1}{4}} \left( 1 + |x - v - \theta t|R^{-\frac{1}{2}} \right)^{-N} MEf_{\theta}(v, 0)
\end{equation}
for any positive integer $N$.
This follows direct stationary phase computation, and details can be found in the proof of
\cite[Lemma 4.1]{tao2003sharp}.

From (\ref{eq:wpd:bd}), direct computation implies that
\begin{equation}
\int |Ef_{\theta, v}(x, 0)|^2 \diff x \lesssim c_{\theta,v}^2,
\quad
\int |Ef_{\theta}(x, 0)|^2 \diff x \lesssim \sum_{ v} c_{\theta,v}^2,
\end{equation}
which, with (\ref{eq:wpd:pf2})
implies (\ref{eq:wpd:l2}).
\end{sproof}

The bound (\ref{eq:wpd:bd}) has the following implication.
\begin{cor}  \label{cor:wpd:bd}
Let $f = \sum_{\theta, v} f_{\theta, v}$ be a wave packet decomposition of size $R$.
Denote $M = \max_{\theta, v}\| c_{\theta, v} \|$.
Then there is
\begin{equation}
\sum_{\theta, v: q \not\in R^{1 + \delta} T_{\theta, v}} |Ef_{\theta, v}(x, t)|
\lesssim_{\delta, N, \lambda}
R^{-N} \| f \|_{L^2}^{\lambda} M^{1 - \lambda}
\end{equation}
for any point $(x, t) \in \R^2$, and $\lambda \in (0 ,1]$.
\end{cor}

\begin{proof}
By (\ref{eq:wpd:bd}), 
for each $\theta \in \R^{-\frac{1}{2}}\Z \bigcap [-1, 1]$, there is
\begin{multline}
\sum_{v: (x, t) \not\in R^{1 + \delta} T_{\theta, v}} |Ef_{\theta, v}(x, t)|
\lesssim_L
\sum_{v: (x, t) \not\in R^{1 + \delta} T_{\theta, v}}
|c_{\theta, v}| R^{-\frac{1}{4}} \left( 1 + |x - v - \theta t|R^{-\frac{1}{2}} \right)^L
\\
\lesssim
M^{1 - \lambda}
\left( \sum_{v} |c_{\theta, v}|^2 \right)^{\frac{\lambda}{2}}
\left( 
\sum_{v: (x, t) \not\in R^{1 + \delta} T_{\theta, v}}
R^{-\frac{1}{4}\left( 1 - \frac{\lambda}{2} \right)^{-1}} \left( 1 + |x - v - \theta t|R^{-\frac{1}{2}} \right)^{L \left( 1 - \frac{\lambda}{2} \right)^{-1} }
\right)^{1 - \frac{\lambda}{2}}  .
\end{multline}
We have
\begin{equation}
\sum_{v: (x, t) \not\in R^{1 + \delta} T_{\theta, v}}
\left( 1 + |x - v - \theta t|R^{-\frac{1}{2}} \right)^{L \left( 1 - \frac{\lambda}{2} \right)^{-1}}
\lesssim
\sum_{n > R^{\frac{\delta}{2}}}
n^{-2L }
\sim
R^{\frac{\delta}{2}\cdot (1-2L ) } ,
\end{equation}
since $v \in R^{\frac{1}{2}} \Z$, and $|x - v - \theta t|R^{-\frac{1}{2}} \gtrsim R^{\frac{\delta}{2}}$.
Summing over $\theta$, by Cauchy--Schwarz inequality there is
\begin{multline}
\sum_{\theta, v: q \not\in R^{1 + \delta} T_{\theta, v}} |Ef_{\theta, v}(x, t)|
\lesssim_{L}
M^{1 - \lambda}
\sum_{\theta}\left( \sum_{v} |c_{\theta, v}|^2 \right)^{\frac{\lambda}{2}}
R^{\frac{\delta}{2}\cdot(1-2L)}
\\
\lesssim
R^{\frac{1}{2} - \frac{\lambda}{4} + \frac{\delta}{2}\cdot(1-2L)} \| f \|_{L^2}^{\lambda}M^{1 - \lambda} .
\end{multline}
Taking $L$ large enough finishes the proof.
\end{proof}

Let's consider wave packets in different scales.
There is a relation when transiting from the wave packet decomposition of one size to another.
\begin{lem}  \label{lemma:wpd:cs}
Let $0 < R_1 < R_2$, and $\left( \frac{R_2}{R_1} \right)^{\frac{1}{2}}$
an integer.
For smooth $f:[-1, 1] \rightarrow \C$, with wave packet decomposition of size $R_2$:
\begin{equation}
f = \sum_{\tau, w} c_{\tau, w} \phi_{\tau, w} = \sum_{\tau, w} f_{\tau, w} = \sum_{\tau} f_{\tau} ,
\end{equation}
then there exists a wave packet decomposition for size $R_1$:
\begin{equation}
f = \sum_{\theta, v} c_{\theta, v} \phi_{\theta, v} = \sum_{\theta, v} f_{\theta, v} = \sum_{\theta} f_{\theta}  ,
\end{equation}
such that
for each $\theta$, $v$, there is
\begin{equation}  \label{eq:wpd:cs:1}
|c_{\theta, v}| \lesssim \left( \frac{R_2}{R_1} \right)^{\frac{1}{4}} \max_{\tau, w}|c_{\tau, w}|  .
\end{equation}
\end{lem}

\begin{proof}
We can do the wave packet decomposition of size $R_1$ in the following way:
we can write $R_2^{-\frac{1}{2}} \Z \bigcap [-1, 1]$ as a disjoint union $\bigcup_{\theta \in R_1^{-\frac{1}{2}} \Z \bigcap [-1, 1]} \mathcal{K}_{\theta}$.
where each $\mathcal{K}_{\theta} \subset \left[ \theta - R_2^{-\frac{1}{2}}, \theta + R_2^{-\frac{1}{2}} \right]$.
Then we let
\begin{equation}
f_{\theta} = \sum_{\tau \in \mathcal{K}_{\theta}} f_{\tau} .
\end{equation}
And the the remaining of the decomposition follows the proof in Proposition \ref{prop:wpd}.

Note that $|\mathcal{K}_{\theta}| \lesssim (R_2/R_1)^{\frac{1}{2}}$.
For any $\theta, v$, from the decomposition, 
there is a neighborhood $\mathcal{A}$ of $v$, such that
\begin{equation}
|c_{\theta, v}|^2 \leq R_1^{\frac{1}{2}} \left( \frac{1}{|\mathcal{A}|} \int_{\mathcal{A}} |Ef_{\theta}|  \right)^2 
\leq \frac{R_1^{\frac{1}{2}}}{|\mathcal{A}|} \int_{\mathcal{A}} |Ef_{\theta}|^2 
\lesssim \sum_{\tau \in \mathcal{K}_{\theta}} \frac{R_2^{\frac{1}{2}}}{|\mathcal{A}|}  \int_{\mathcal{A}} |Ef_{\tau}|^2  .
\end{equation}
At each point $|Ef_{\tau}|^2 \sim \sum_{w} |c_{\tau, w}|^2$, thus is bounded by $R_2^{-\frac{1}{2}} \max_w |c_{\tau, w}|^2$.
Then
each $\frac{R_2^{\frac{1}{2}}}{|\mathcal{A}|}  \int_{\mathcal{A}} |Ef_{\tau}|^2$
is bounded by $\max_{w} |c_{\tau, w}|^2 $.
Summing over $\tau$ leads to (\ref{eq:wpd:cs:1}).
\end{proof}



\bigskip
\section{Main Results}  \label{sec:res}

\subsection{The Problem and Obtained Results}

For any wave packet decomposition (of size $R$) of smooth function $f$ ,
denote
\begin{equation}  \label{eq:mainieq:defs}
S = \frac{\max_{\theta, v} \|f_{\theta, v}\|_{L^2}   }{\|f\|_{L^2}} \sim
\frac{\max_{\theta, v} |c_{\theta, v}| }{\|f\|_{L^2}}  .
\end{equation}
We would like to find all $\alpha, \beta \geq 0$,
such that the following inequality always holds for any $\epsilon > 0$, sufficiently small (relying on $p, \alpha, \beta$):
\begin{equation}  \label{eq:mainieq}
\|Ef\|_{L^p(B_R)} \leq C(\alpha, \beta, p, \epsilon) R^{\alpha + \epsilon}S^{\beta - \epsilon} \|f\|_{L^2}  ,
\end{equation}
where $C(\alpha, \beta, p, \epsilon)$ is a constant.
We also denote $M = \max_{\theta, v} \|f_{\theta, v}\|_{L^2} = S \|f\|_{L^2}$ to be the size of largest wave packet.
We will use both notations ($S$ and $M$) in our proofs.

Let us emphasize that the decomposition is \emph{not unique}, and can lead to different $S$.
The constant $C(\alpha, \beta, p, \epsilon)$ is independent of the function $f$ and the choice of the decomposition.


Also, if for some $p$, the pair $\alpha_1$, $\beta_1$, and $\alpha_2$, $\beta_2$ make (\ref{eq:mainieq}) hold, respectively, then for any $\lambda \in [0, 1]$, so does the pair $\lambda \alpha_1 + (1 - \lambda) \alpha_2$, $\lambda \beta_1 + (1 - \lambda) \beta_2$.

The results to be presented in this text are as following:
\begin{thm}   \label{thm:main}
For $p \in [2, 6]$, the inequality (\ref{eq:mainieq}) holds for any smooth function $f: [-1, 1] \rightarrow \C$, and any of its wave packet decomposition, if $p, \alpha, \beta$ satisfies
\begin{equation}  \label{eq:thm:main}
\begin{cases}
\beta \leq 1\\
4 p\alpha + p \geq 6 \\
2 p\alpha - p\beta + p \geq 4 \\
4 \alpha - \beta \geq 0 \\
12 p \alpha - 4 p\beta + 3p \geq 14
\end{cases} ,
\end{equation}
and only if $p, \alpha, \beta$ satisfies
\begin{equation}  \label{eq:thm:main:rv}
\begin{cases}
\beta \leq 1 \\
4 p\alpha + p \geq 6 \\
2 p\alpha - p\beta + p \geq 4 \\
4 \alpha - \beta \geq 0
\end{cases}.
\end{equation}
\end{thm}
The results are better illustrated via a diagram of the parameter space (see Figure \ref{fig:1}):
we prove that (\ref{eq:mainieq}) holds in the convex hull supported by $M, N, X, Y, U, V, W$, and its extension in the $\alpha p$ direction;
whether (\ref{eq:mainieq}) holds inside the tetrahedron $UVWF$ is still open;
and for the remaining areas between $p = 2$ and $p = 6$ we've shown that (\ref{eq:mainieq}) cannot be true.

It's worth noting that the above result means that we've found all possible pairs $\alpha$, $\beta$ for $2 \leq p \leq 4$ and $p = 6$;
and for $5 \leq p < 6$, the bounds we have cannot be improved when $S \geq R^{- \frac{1}{4}}$.

\begin{figure}  
\centering
    \includegraphics[width=0.8\textwidth]{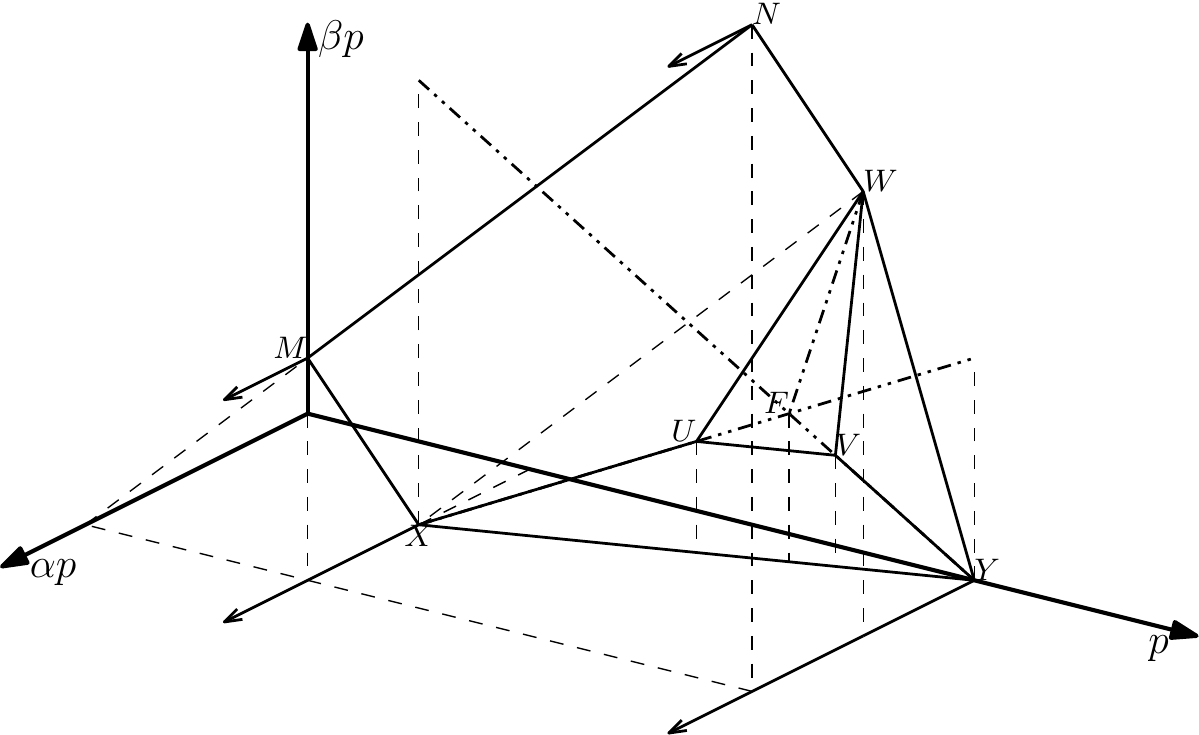}
    \caption{The parameter space} \label{fig:1}
\end{figure}


\subsection{Initial Steps and the Structure of Proof}
There are a few observations which make the proof much clearer.

First, it is obvious that if (\ref{eq:mainieq}) holds for $p, \alpha, \beta$, then it also holds for $p, \alpha', \beta$, for $\alpha' > \alpha$.

Besides, if (\ref{eq:mainieq}) holds for some $p_1, \alpha_1, \beta_1$, and $p_2, \alpha_2, \beta_2$, then for any $\lambda \in [0, 1]$, (\ref{eq:mainieq}) holds for $p_1\lambda + p_2(1 - \lambda), \frac{\alpha_1\lambda + \alpha_2(1 - \lambda)}{p_1\lambda + p_2(1 - \lambda)}, \frac{\beta_1\lambda + \beta_2(1 - \lambda)}{p_1\lambda + p_2(1 - \lambda)}$.
This directly follows H\"{o}lder's inequality.

Also, note that inside $B_R$, there are \emph{essentially} $\lesssim R$ wave packets, and it suffices to consider the cases where $R^{-\frac{1}{2}} \lesssim S \lesssim 1$.
In other words, we have the following Lemma:
\begin{lem}  \label{lemma:ext}
If (\ref{eq:mainieq}) holds for some $p, \alpha, \beta$, then for any $\lambda > 0$, it also holds for $p, \alpha + \lambda, \beta + 2\lambda$, with $\beta + 2\lambda \leq 1$.
\end{lem}
\begin{proof}
Take $\delta = \frac{\epsilon}{2}$.
Let $f' = \sum_{\theta, v: R^{1 + \delta}T_{\theta, v} \bigcap B_R \neq \emptyset} f_{\theta, v}$, and this gives a wave packet decomposition of $f'$.
Then $\| f' \|_{L^2}^2 \sim \sum_{\theta, v: R^{1 + \delta}T_{\theta, v} \bigcap B_R \neq \emptyset} |c_{\theta, v}|^2 \lesssim R^{1 + \frac{\delta}{2}} M^2$.
By Corollary \ref{cor:wpd:bd}, 
there is
\begin{equation}
\| Ef - Ef' \|_{L^p B_R} \lesssim_{\delta, N, \lambda} R^{-N} S^{\beta + 2 \lambda - \delta} \| f - f'\|_{L^2} \lesssim R^{-N} S^{\beta + 2 \lambda - \delta} \| f\|_{L^2}
\end{equation}
for any $N > 0$, since $\beta + 2\lambda - \delta \in [0, 1)$ (when $\delta$ sufficiently small).
Using the fact that (\ref{eq:mainieq}) holds for $p, \alpha, \beta$, we conclude
\begin{multline}
\| Ef \|_{L^p(B_R)} \lesssim_{\delta} \| Ef' \|_{L^p(B_R)} + R^{\alpha + \delta}S^{\beta + 2 \lambda - \delta}\| f \|_{L^2}
\\
\lesssim_{\delta}
R^{\alpha + \delta} M^{\beta - \delta} \| f' \|_{L^2}^{1 + \delta - \beta}
+
R^{\alpha + \delta}S^{\beta + 2 \lambda}\| f \|_{L^2}
\\
\lesssim_{\delta}
R^{\alpha + \delta  + \lambda \left( 1 + \frac{\delta}{2} \right) } M^{\beta - \delta + 2 \lambda} \| f' \|_{L^2}^{1 + \delta - \beta - 2\lambda}
+
R^{\alpha + \lambda + \delta}S^{\beta + 2 \lambda}\| f \|_{L^2}
\lesssim
R^{\alpha + \lambda + \epsilon} S^{\beta + 2 \lambda - \epsilon} \| f \|_{L^2}.
\end{multline}
\end{proof}

Then
it suffices to prove (\ref{eq:mainieq}) for the end points:
\begin{itemize}
\item $p = 2$, $\alpha = \frac{1}{2}$, $\beta = 0$. This corresponds to point $X = (2, 1, 0)$ in Figure \ref{fig:1}, and is part of (\ref{eq:trf}).
\item $p = 4$, $\alpha = \frac{1}{8}$, $\beta = \frac{1}{4}$. This corresponds to point $U = \left(4, \frac{1}{2}, 1 \right)$ in Figure \ref{fig:1},
and is proved in Section \ref{sec:l4}, with arithmetic approaches.
\item $p = 5$, $\alpha = \frac{1}{20}$, $\beta = \frac{1}{5}$. 
This corresponds to point $V = \left(5, \frac{1}{4}, 1 \right)$ in Figure \ref{fig:1},
and is proved by polynomial partitioning in Section \ref{sec:pp}.
\item $p = 6$, $\alpha = 0$, $\beta = 0$.
This corresponds to point $Y = \left(6, 0, 0 \right)$ in Figure \ref{fig:1},
and is also part of (\ref{eq:trf}).
\item $p = 6$, $\alpha = \frac{1}{6}$, $\beta = \frac{2}{3}$.
This corresponds to point $W = \left(6, 1, 4 \right)$ in Figure \ref{fig:1},
and is proved using $l^2$ Decoupling Theorem in Section \ref{sec:dc}.
\end{itemize}

In Section \ref{sec:em} we give examples, showing that (\ref{eq:thm:main:rv}) is a necessary condition for (\ref{eq:mainieq}) to hold.
In Section \ref{sec:cj} we make the conjecture, which implies that (\ref{eq:thm:main:rv}) is also sufficient.






\bigskip
\section{Explicit Computation: Proof for $p=4$}  \label{sec:l4}

In this section we prove the bound for $p = 4$ by direct computation, writing the integral in $L^4$ norm as a sum of integrals of the wave packets.
By doing this we can exploit the fact that different wave packets are far away in either physical or frequency space.

We point out that bounds given in this approach are always sharp; but it can be used only when $p$ is an even integer.

\begin{prop}  \label{prop:pc}
When $p = 4$,
(\ref{eq:mainieq}) holds for $\alpha = \frac{1}{8}$
and $\beta = \frac{1}{4}$   .
\end{prop}

\begin{proof}
By Plancherel Theorem there is
\begin{equation}
\begin{split}
& \| Ef \|_{L^4(B_R)} = \int_{B_R} \left|  \sum_{\theta}  E{f}_{\theta} \right|^4 
\lesssim 
\int \left| \eta_R \sum_{\theta}  E{f}_{\theta} \right|^4 
=
\int \left( \eta_R \sum_{\theta}  E{f}_{\theta} \right)^2
\left( \overline{\eta_R \sum_{\theta}  E{f}_{\theta} }\right)^2
\\
= & \sum_{\theta_1, \theta_2, \theta_3, \theta_4} \int \eta_R E{f}_{\theta_1} E{f}_{\theta_2}\cdot \overline{\eta_R E{f}_{\theta_3} E{f}_{\tau_4}} \\
= & \sum_{\theta_1, \theta_2, \theta_3, \theta_4} \int \left( \widehat{\eta_R} \ast \widehat{E{f}_{\theta_1}} \ast \widehat{E{f}_{\theta_2}} \right) \cdot \left( \overline{\widehat{\eta_R}} \ast \overline{\widehat{E{f}_{\theta_3}}} \ast \overline{\widehat{E{f}_{\theta_4}}} \right) \\
= & \sum_{\theta_1, \theta_2, \theta_3, \theta_4} \int \left( \widehat{\eta_R} \ast f^*_{\theta_1}\diff \sigma \ast f^*_{\theta_2}\diff \sigma \right) \cdot \left( \overline{\widehat{\eta_R}} \ast \overline{ f^*_{\theta_3} }\diff \sigma \ast \overline{ f^*_{\theta_4}}\diff \sigma  \right) .
\end{split}
\end{equation}
The convolution $\widehat{\eta_R} \ast f^*_{\theta_1}\diff \sigma \ast f^*_{\theta_2}\diff \sigma$ is supported in the $R^{-1}$ neighborhood of the Minkowski sum of the supports of $f^*_{\theta_1}$ and $f^*_{\theta_2}$;
and the convolution $\overline{\widehat{\eta_R}} \ast \overline{ f^*_{\theta_3} }\diff \sigma \ast \overline{ f^*_{\theta_4}}\diff \sigma$ is supported in the $R^{-1}$ neighborhood of the Minkowski sum of the supports of $f^*_{\theta_3}$ and $f^*_{\theta_4}$.
The integral
\begin{equation}   \label{eq:dp:pf0}
\int \left( \widehat{\eta_R} \ast f^*_{\theta_1}\diff \sigma \ast f^*_{\theta_2}\diff \sigma \right) \cdot \left( \overline{\widehat{\eta_R}} \ast \overline{ f^*_{\theta_3} }\diff \sigma \ast \overline{ f^*_{\theta_4}}\diff \sigma  \right)
\end{equation}
is none-zero only if the supports of $\widehat{\eta_R} \ast f^*_{\theta_1}\diff \sigma \ast f^*_{\theta_2}\diff \sigma$ and $\overline{\widehat{\eta_R}} \ast \overline{ f^*_{\theta_3} }\diff \sigma \ast \overline{ f^*_{\theta_4}}\diff \sigma $ have non-empty intersection,
and only if there exists
$x_i$ in the $R^{-\frac{1}{2}}$-neighborhood of $\theta_i$, $i=1,2,3,4$,
such that
\begin{equation}  \label{eq:dp:pf1}
\begin{cases}
\left| x_1 + x_2 - x_3 - x_4 \right| & \leq 2R^{-1}  , \\
\left| x_1^2 + x_2^2 - x_3^2 - x_4^2 \right| & \leq 2R^{-1}  .
\end{cases}
\end{equation}
The first inequality in (\ref{eq:dp:pf1}) implies that
\begin{equation}
\left| (x_1 + x_2 - x_3)^2 - x_4^2 \right| \leq 8R^{-1}  ,
\end{equation}
and we have
\begin{equation}
2\left| x_1 - x_3 \right|\left| x_2 - x_3 \right| = \left| x_1^2 + x_2^2 - x_3^2 - (x_1 + x_2 - x_3)^2 \right| \leq 9R^{-1}  ,
\end{equation}
thus
\begin{equation}
\min \{ |x_1 - x_3|, |x_2 - x_3| \} \leq 3R^{-\frac{1}{2}}  .
\end{equation}
This implies that $\min \{ |\theta_1 - \theta_3|, |\theta_2 - \theta_3| \} \leq 4R^{-\frac{1}{2}}$;
and with the first inequality in (\ref{eq:dp:pf1}) we conclude that either $|\theta_1 - \theta_3|, |\theta_2 - \theta_4| \leq 4R^{-\frac{1}{2}}$, or $|\theta_2 - \theta_3|, |\theta_1 - \theta_4| \leq 4R^{-\frac{1}{2}}$.

Then for fixed $(\theta_1, \theta_2)$,
there are $\lesssim 1$ pairs of $(\tau_3, \tau_4)$,
such that (\ref{eq:dp:pf0}) is non-zero.

Now we have
\begin{multline}  \label{eq:pf002}
\int_{B_R} \left|  \sum_{\theta}  E{f}_{\theta} \right|^4  \lesssim
\sum_{\theta_1, \theta_2} \int \left| \eta_R E{f}_{\theta_1} E{f}_{\theta_2} \right|^2 
\\
=
\sum_{\theta_1, \theta_2} \sum_{u_1, u_2 \in R^{\frac{1}{2}}\Z}
\int_{RT_{\theta_1, u_1} \bigcap RT_{\theta_2, u_2}}
\left|\eta_R \sum_{v_1, v_2 \in R^{\frac{1}{2}}\Z } Ef_{\theta_1, v_1} Ef_{\theta_2, v_2}\right|^2 
\\
\lesssim
\sum_{\theta_1, \theta_2} \sum_{v_1, v_2, u_1, u_2 \in R^{\frac{1}{2}}\Z}
\int_{RT_{\theta_1, u_1} \bigcap RT_{\theta_2, u_2}}
|\eta_R Ef_{\theta_1, v_1}Ef_{\theta_2, v_2}|^2 (1 + |u_1 - v_1| R^{-\frac{1}{2}} )^2 (1 + |u_2 - v_2| R^{-\frac{1}{2}} )^2,
\end{multline}
and by (\ref{eq:wpd:bd}), (\ref{eq:pf002}) is further bounded by
\begin{equation}  \label{eq:pc:pf2}
\sum_{\theta_1, \theta_2} \sum_{v_1, v_2, u_1, u_2 \in R^{\frac{1}{2}}\Z} 
\frac{R^{-1}|c_{\theta_1, v_1}c_{\theta_2, v_2}|^2 \int_{RT_{\theta_1, u_1} \bigcap RT_{\theta_2, u_2}} |\eta_R |^2  }{(1 + |u_1 - v_1| R^{-\frac{1}{2}} )^{10} (1 + |u_2 - v_2| R^{-\frac{1}{2}}  )^{10}}.
\end{equation}
Taking the square of (\ref{eq:pc:pf2}), we get
\begin{multline}  \label{eq:pf003}
\sum_{\substack{\theta_1, \theta_2, \theta_3, \theta_4 \\ v_1, v_2, u_1, u_2 \\ v_3, v_4, u_3, u_4 } }
\frac{R^{-2}|c_{\theta_1, v_1}c_{\theta_2, v_2}c_{\theta_3, v_3}c_{\theta_4, v_4}|^2
\left( \int_{RT_{\theta_1, u_1} \bigcap RT_{\theta_2, u_2}} |\eta_R |^2  \right)
\left( \int_{RT_{\theta_3, u_3} \bigcap RT_{\theta_4, u_4}} |\eta_R |^2  \right)
}{
\left( (1 + |u_1 - v_1| R^{-\frac{1}{2}} ) (1 + |u_2 - v_2| R^{-\frac{1}{2}}  ) (1 + |u_3 - v_3| R^{-\frac{1}{2}} ) (1 + |u_4 - v_4| R^{-\frac{1}{2}}  ) \right)^{10}}
\\
\lesssim
\sum_{\substack{\theta_1, \theta_2, \theta_3, \theta_4 \\ v_1, v_2, u_1, u_2 \\ v_3, v_4, u_3, u_4 } }
\frac{R^{-2}|c_{\theta_1, v_1}c_{\theta_3, v_3}c_{\theta_4, v_4}^2|^2
\left( \int_{RT_{\theta_1, u_1} \bigcap RT_{\theta_2, u_2}} |\eta_R |^2  \right)^2
}{
\left( (1 + |u_1 - v_1| R^{-\frac{1}{2}} ) (1 + |u_2 - v_2| R^{-\frac{1}{2}}  ) (1 + |u_3 - v_3| R^{-\frac{1}{2}} ) (1 + |u_4 - v_4| R^{-\frac{1}{2}}  ) \right)^{10}}
\\
+
\sum_{\substack{\theta_1, \theta_2, \theta_3, \theta_4 \\ v_1, v_2, u_1, u_2 \\ v_3, v_4, u_3, u_4 } }
\frac{R^{-2}|c_{\theta_1, v_1}c_{\theta_2, v_2}^2c_{\theta_3, v_3}|^2
\left( \int_{RT_{\theta_3, u_3} \bigcap RT_{\theta_4, u_4}} |\eta_R |^2  \right)^2
}{
\left( (1 + |u_1 - v_1| R^{-\frac{1}{2}} ) (1 + |u_2 - v_2| R^{-\frac{1}{2}}  ) (1 + |u_3 - v_3| R^{-\frac{1}{2}} ) (1 + |u_4 - v_4| R^{-\frac{1}{2}}  ) \right)^{10}}
\\
=
\sum_{\substack{\theta_1, \theta_2, \theta_3, \theta_4 \\ v_1, v_2, u_1, u_2 \\ v_3, v_4, u_3, u_4 } }
\frac{2R^{-2}|c_{\theta_1, v_1}c_{\theta_3, v_3}c_{\theta_4, v_4}^2|^2
\left( \int_{RT_{\theta_1, u_1} \bigcap RT_{\theta_2, u_2}} |\eta_R |^2  \right)^2
}{
\left( (1 + |u_1 - v_1| R^{-\frac{1}{2}} ) (1 + |u_2 - v_2| R^{-\frac{1}{2}}  ) (1 + |u_3 - v_3| R^{-\frac{1}{2}} ) (1 + |u_4 - v_4| R^{-\frac{1}{2}}  ) \right)^{10}} .
\end{multline}
Summing over $v_2, u_3, u_4$, (\ref{eq:pf003}) is bounded by
\begin{equation}  \label{eq:pc:pf3}
\sum_{\substack{\theta_1, \theta_2, \theta_3, \theta_4 \\ v_1, u_1, u_2 \\ v_3, v_4} }
\frac{R^{-2}|c_{\theta_1, v_1}c_{\theta_3, v_3}c_{\theta_4, v_4}^2|^2
\left( \int_{RT_{\theta_1, u_1} \bigcap RT_{\theta_2, u_2}} |\eta_R |^2  \right)^2
}{
\left( 1 + |u_1 - v_1| R^{-\frac{1}{2}}  \right)^{10}} .
\end{equation}
For each $\theta_1, u_1, \theta_2$, there is
\begin{equation}
\sum_{u_2} \int_{RT_{\theta_1, u_1} \bigcap RT_{\theta_2, u_2}} |\eta_R |^2
\sim
\int_{RT_{\theta_1, u_1}} |\eta_R |^2
\lesssim R^{\frac{3}{2}},
\end{equation}
by (\ref{eq:sf:bd}).
Also, for any $\theta_1 \neq \theta_2, u_1, u_2$, there is
\begin{equation}
\int_{RT_{\theta_1, u_1} \bigcap RT_{\theta_2, u_2}} |\eta_R |^2
\lesssim
\int_{RT_{\theta_1, u_1} \bigcap RT_{\theta_2, u_2}} 1 \lesssim R^{\frac{3}{2}} \left( |\theta_1 - \theta_2| R^{\frac{1}{2}} \right)^{-1}. 
\end{equation}
We conclude that, for fixed $\theta_1, u_1$, there is
\begin{equation}  \label{eq:pc:pf5}
\sum_{\theta_2, u_2} \left( \int_{RT_{\theta_1, u_1} \bigcap RT_{\theta_2, u_2}} |\eta_R |^2 \right)^2
\lesssim
R^3 \sum_{\theta_2} \left( |\theta_1 - \theta_2| R^{\frac{1}{2}} \right)^{-1}
\sim R^3 \log(R).
\end{equation}
Plugging (\ref{eq:pc:pf5}) into (\ref{eq:pc:pf3}) and summing over $u_1$, we see that (\ref{eq:pc:pf3}) is bounded by
\begin{equation}
\sum_{\substack{\theta_1, \theta_3, \theta_4 \\ v_1, v_3, v_4} }
R\log(R)\left|c_{\theta_1, v_1}c_{\theta_3, v_3}c_{\theta_4, v_4}^2\right|^2
\leq
R\log(R)M^2\left( \sum_{\theta, u} \left|c_{\theta, v}\right|^2 \right)^3
\sim R\log(R)M^2 \| f\|_{L^2}^6.
\end{equation}
This implies that (\ref{eq:pc:pf2}) $\lesssim_{\epsilon} R^{\frac{1}{2} + 4\epsilon} M \| f \|_{L^2}^3 = R^{\frac{1}{2} + 4\epsilon} S \| f \|_{L^2}^4$, and we finish the proof.
\end{proof}

\bigskip
\section{Polynomial Partitioning: Proof for $p=5$}  \label{sec:pp}

In this section we use the approach of polynomial partitioning.
It was introduced by Larry Guth and Nets Katz in \cite{guth2010erdos}, 
and further used by Larry Guth to the restriction problem
(see \cite{guth2016restriction} \cite{1603.04250}).
We exploit this method in this problem, and prove the following result.
\begin{prop}   \label{prop:pm}
For $p = 5$,
(\ref{eq:mainieq}) holds for $\alpha = \frac{1}{20}$ and $\beta = \frac{1}{5}$.
\end{prop}

The main idea is to use a polynomial of degree $D$ to cut $\R^2$ into $D^2$ cells, each has the same $L^p$ norm,
and bound each cell using an induction argument.
As each wave packet is essentially supported in a tube, which like a line, and thus it is expected to intersect the polynomial curve $\sim D$ times.
Then we expect that on average, each cell intersects $\sim D^{-1}$ of all the wave packets.

We use the following version of polynomial partitioning in our proof, which also follows Stone-Tukey ham sandwich theorem \cite{stone1942}.
\begin{thm} \protect{\cite[Theorem 0.6]{guth2016restriction}}  \label{thm:pp}
For any non-negative $L^1$ function $W$ on $\R$,
and degree $D \geq 1$,
there is a non-zero polynomial $\mathcal{F}$ of degree at most $D$,
whose zero set is denoted as $Z(\mathcal{F})$,
and $\R^2 \backslash Z(\mathcal{F})$ is a union of $\sim D^2$
disjoint cells $O_i$, with all $\int_{O_i} W$ equal.
\end{thm}

Let us point out that actually the tubes are different from lines: it can intersect arbitrarily many cells.
To overcome this, we consider a ``strip'', which is a $R^{\frac{1}{2} + \delta}$-neighborhood of of the curve $Z(\mathcal{F})$.
Now we need to bound the $L^p$ norm on the strip as well.
In some cases the $L^p$ norm on the strip makes up a significant portion of the total $L^p$ in $B_R$.
For each tube, when intersecting the strip, it can be either parallel or transverse to the strip.
We analyze these situations separately.

The following theorem of Wongkew will be used to estimate the area of the strip.
\begin{thm} \cite{wongkew1993volumes}  \label{thm:pm:wv}
Let $\mathcal{F}$ be a non-zero polynomial of degree $D$ on $\R^2$,
and $\mathcal{N}$ the $\rho$-neighborhood of $Z(\mathcal{F})$.
For $B_R$ any ball of radius $R$,
there is
\begin{equation}
\Vol(B_R\bigcap \mathcal{N}) \lesssim D\rho R  .
\end{equation}
\end{thm}

\begin{proof}[Proof of Proposition \ref{prop:pm}]
In the wave packet decomposition $f = \sum_{\theta, v} f_{\theta, v} = \sum_{\theta, v} c_{\theta, v} \phi_{\theta, v}$ of size $R$, denote $K(f)$ to be the number of non-zero $c_{\theta, v}$.
It suffices to prove the case where $K(f)$ is finite.
Indeed, by Corollary \ref{cor:wpd:bd},
there is
\begin{equation}
\int_{B_R} |Ef|^5 
\lesssim
\int_{B_R} \left| \sum_{\substack{\theta, v: RT_{\theta, v} \bigcap B_{R^2} \neq \emptyset} } Ef_{\theta, v} \right|^5
+
R^{-10} \| f \|_{L^2}^5 .
\end{equation}
The first term has only finitely many nonzero wave packets, while the second is obviously bounded.

Now we do induction on both the radius $R$ and $K(f)$.

In the case where $K(f) = 1$, 
we have $S = 1$, and (\ref{eq:mainieq}) follows Theorem \ref{thm:str}.

Now consider finite $K(f)$.
Denote $\chi_{B_R}$ to be the indicator function of $B_R$.
For any positive integer $D$,
by applying Theorem \ref{thm:pp} to $\chi_{B_R} |f|^5$,
there exists a non-zero polynomial $\mathcal{F}_D$, of degree at most $D$,
and $\R^2 \backslash Z(\mathcal{F}_D)$ is a union of $\sim D^2$ disjoint cells $\tilde{O}_i$,
such that all $\int_{\tilde{O}_i \bigcap B_R} |f|^5$ equal.

Let $\delta = \frac{\epsilon}{20}$ .
Denote $\mathcal{N}$ to be the $R^{\frac{1}{2} + \delta}$-neighborhood of $Z(\mathcal{F}_D)$ inside $B_R$,
and $O_i = \tilde{O}_i \bigcap B_R \backslash \mathcal{N}$.
Then there is
\begin{equation}
\| Ef \|_{L^5(B_R)}^5  =  \| Ef \|_{L^5(\mathcal{N})}^5 + \sum_{{O}_i}\| Ef \|_{L^5({O}_i)}^5  .
\end{equation}

Now consider the following two situations
\begin{case} $\sum_{{O}_i} \| Ef \|_{L^5({O}_i)}^5 > \frac{3}{4} \| Ef \|_{L^5(B_R)}^5 = \frac{3}{4} \sum_{\tilde{O}_i} \| Ef \|_{L^5(\tilde{O}_i)}^5 $.

Then for half of all ${O}_i$, there is
\begin{equation}
\| Ef \|_{L^5({O}_i)}^5 \geq \frac{1}{2} \| Ef \|_{L^5(\tilde{O}_i)}^5 \geq C_1 D^{-2} \| Ef \|_{L^5(B_R)}^5  ,
\end{equation}
for some constant $C_1$.



For each $\theta, v$, let $\mathcal{L}_{\theta, v}$ be the line $x = v + \theta t$;
then $R^{1 + 2\delta}T_{\theta, v}$ is the $R^{\frac{1}{2} + \delta}$ neighborhood of $\mathcal{L}_{\theta, v}$.
For any cell $O_i$,
if $R^{1 + 2\delta}T_{\theta, v} \bigcap {O}_i \neq \emptyset$,
we must have  $\mathcal{L}_{\theta, v} \bigcap \tilde{O}_i \neq \emptyset$.
However, since $\mathcal{F}_D$ is of degree $D$,
$\mathcal{L}_{\theta, v}$ has at most $D$ intersections with the curve $Z(\mathcal{L}_{\theta, v})$;
thus $R^{1 + 2\delta}T_{\theta, v}$ intersects at most $D$ of all ${O}_i$.
In particular, $\lesssim D$ of all ${O}_i$ intersect each of $R^{1 + 2\delta}T_{\theta, v}$.

Also, by (\ref{eq:wpd:l2}),
there is
\begin{equation}
\sum_{{O}_i} \| f_{{O}_i} \|_{L^2}^2
\sim \sum_{{O}_i} \sum_{\substack{\theta, v \\ R^{1 + 2\delta}T_{\theta, v} \bigcap {O}_i \neq \emptyset }} | c_{\theta, v} |^2  
\leq D \sum_{\theta, v} \| c_{\theta, v} |^2
\sim D \| f \|_{L^2}^2  ,
\end{equation}
then more than $\frac{3}{4}$ of all ${O}_i$ satisfy that
\begin{equation}
\| f_{{O}_i} \|_{L^2}^2 \leq C_2 D^{-1} \| f \|_{L^2}^2   ,
\end{equation}
for some constant $C_2$.

Now we can find a cell ${O}_{i_o}$, which does not intersect every $R^{1 + 2\delta}T_{\theta, v}$,
and satisfies
\begin{equation}
\| Ef \|_{L^5({O}_{i_o})}^5 \geq C_1D^{-2} \| Ef \|_{L^5(B_R)}^5  ,
\quad
\| f_{{O}_{i_o}} \|_{L^2}^2 \leq C_2 D^{-1} \| f \|_{L^2}^2   .
\end{equation}

We denote
\begin{equation}
f_{{O}_{i_o}} = \sum_{\theta, v :\; R^{1 + 2\delta}T_{\theta, v} \bigcap {O}_{i_o} \neq \emptyset } f_{\theta, v} .
\end{equation}
Then there is 
\begin{equation}
\int_{O_{i_o} } |Ef|^5 = \int_{O_{i_o} } \left| \sum_{\theta, v} Ef_{\theta, v} \right|^5
\lesssim \int_{O_{i_o} } |Ef_{O_{i_o} }|^5 +
\int_{O_{i_o} } |E(f - f_{O_{i_o} })|^5
 .
\end{equation}
Since $f - f_{O_{i_o} }$ is the sum of all $f_{\theta, v}$ where
each $R^{1 + 2\delta}T_{\theta, v} \bigcap {O}_{i_o} = \emptyset$,
by Corollary \ref{cor:wpd:bd} there is
\begin{equation}
\| Ef \|_{L^5(O_{i_o} )} = \left( \int_{O_{i_o} } |Ef|^5 \right)^{\frac{1}{5}}
\lesssim 
C_{\delta}
\| Ef_{{O}_{i_o} } \|_{L^5(B_R)}  + R^{-2} \| f \|_{L^2},
\end{equation}
where $C_{\delta}$ is a constant relying on $\delta$.

Using the induction hypothesis for $f_{{O}_{i_o}}$, 
we obtain
\begin{equation}
\begin{split}
& \| Ef \|_{L^5(B_R)} \leq C_1^{-\frac{1}{5}} D^{\frac{2}{5}} \| Ef \|_{L^5({O}_{i_o})}  \\
\leq &
C_1^{-\frac{1}{5}} D^{\frac{2}{5}} \left(
C_{\delta}
\| Ef_{{O}_{i_o}} \|_{L^5(B_R)}  + R^{-2} \| f \|_{L^2} \right) \\
\leq & 
C_1^{-\frac{1}{5}} D^{\frac{2}{5}} \left(
C_{\delta}
C\left(\frac{1}{20}, \frac{1}{5}, 5, \epsilon\right)R^{\frac{1}{20} + \epsilon} M^{\frac{1}{5} - \epsilon} \| f_{{O}_{i_o}} \|_{L^2}^{\frac{4}{5} + \epsilon}  + R^{-2} \| f \|_{L^2} \right) \\
\leq &
C_1^{-\frac{1}{5}} 
C_2^{\frac{2}{5} + \frac{\epsilon}{2}}
C_{\delta}
C\left(\frac{1}{20}, \frac{1}{5}, 5, \epsilon\right)R^{\frac{1}{20} + \epsilon} M^{\frac{1}{5} - \epsilon} 
D^{ - \frac{\epsilon}{2}} \| f \|_{L^2}^{\frac{4}{5} + \epsilon}
+ 
C_1^{-\frac{1}{5}} D^{\frac{2}{5}} R^{-2} \| f \|_{L^2} 
.
\end{split}
\end{equation}
Now one can take $D$ large enough (relying on $\epsilon$, $\delta$ and $C_o$ ),
such that when $R$ large enough (relying on $D$),
there is
\begin{equation}
\| Ef \|_{L^5(B_R)} \leq C\left(\frac{1}{20}, \frac{1}{5}, \epsilon, 5\right)R^{\frac{1}{20} + \epsilon} M^{\frac{1}{5} - \epsilon} \| f \|_{L^2}^{\frac{4}{5} + \epsilon} 
=
C\left(\frac{1}{20}, \frac{1}{5}, 5, \epsilon\right)R^{\frac{1}{20} + \epsilon} S^{\frac{1}{5} - \epsilon} \| f \|_{L^2}
 .
\end{equation}

\end{case}

\begin{case} $\| Ef \|_{L^5(\mathcal{N})}^5 > \frac{1}{4} \| Ef \|_{L^5(B_R)}^5 $.

For any $\theta, v$,
consider the tube $2R^{\frac{1}{2} + \delta}T_{\theta, v}$.
The boundary of the tube consists of two lines, and the curve $Z(\mathcal{F}_D$ can cross it $\lesssim D$ times;
and also $Z(\mathcal{F}_D$ can cross the boundary of $B_R$ for $\lesssim D$ times.
This means that $2R^{\frac{1}{2} + \delta}T_{\theta, v} \bigcap Z(\mathcal{F}_D) \bigcap B_r$ consists of $\lesssim D$ segments.
Denote them as $Z_1, \cdots, Z_r$; and for each $1 \leq i \leq r$, denote $\mathcal{N}_i$ to be the $R^{\frac{1}{2} + \delta}$-neighborhood of $Z_i$.


For any point $a \in R^{\frac{1}{2} + \delta} T_{\theta, v} \bigcap \mathcal{N}$,
there is $b \in Z(\mathcal{F}_D)$,
such that $|a - b| \leq R^{\frac{1}{2} + \delta}$.
This means that $b \in 2R^{\frac{1}{2} + \delta}T_{\theta, v}$,
$b \in Z_i$ for some $1 \leq i \leq r$,
and then $a \in \mathcal{N}_i$.
Thus there is
\begin{equation}  \label{eq:pp:pf2}
R^{\frac{1}{2} + \delta} T_{\theta, v} \bigcap \mathcal{N} \subset \bigcup_{1\leq i \leq r} \mathcal{N}_i  .
\end{equation}

Let $G$ be a constant (to be determined later).
For any $\theta, v$
and a corresponding set $\mathcal{N}_i$,
if $\mathcal{N}_i$ has diameter $> \frac{R}{G}$,
we say that $\mathcal{N}_i$ is \emph{tangent} to $\mathcal{N}$;
otherwise it is \emph{transverse} to $\mathcal{N}$.
Let $\Omega_{\theta, v, tang}$ be the collection of all tangent sets $\mathcal{N}_i$,
and $\Omega_{\theta, v, trans}$ be the collection of all transverse sets $\mathcal{N}_i$.

Now divide $B_R$ into $\sim G^2$  squares, each of size $\frac{R}{G} \times \frac{R}{G}$,
and denote the collection of these squares as $\mathbb{S}$.
For each $Q \in \mathbb{S}$, define $f_{trans}, f_{non} : [0, 1] \rightarrow \C$ as following:
\begin{equation}
f_{Q, trans} = \sum_{\substack{\theta, v : \\ \exists \mathcal{N}_i \in \Omega_{\theta, v, trans}, \\ \mathcal{N}_i \bigcap Q \neq \emptyset}} f_{\theta, v}   , \quad
f_{Q, non} = \sum_{\substack{\theta, v : \\ \forall \mathcal{N}_i \in \Omega_{\theta, v, trans}, \\ \mathcal{N}_i \bigcap Q = \emptyset}} f_{\theta, v}   .
\end{equation}
Obviously, for each $Q$ there is $f = f_{Q, trans} + f_{Q, non}$, and
\begin{equation}
\| Ef \|_{L^5(\mathcal{N})}^5
\lesssim 
\sum_{Q \in \mathbb{S}} \| Ef_{Q, trans} \|_{L^5(\mathcal{N} \bigcap Q)}^5 +
\sum_{Q \in \mathbb{S}} \| Ef_{Q, non} \|_{L^5(\mathcal{N} \bigcap Q)}^5  .
\end{equation}

First consider $\sum_{Q \in \mathbb{S}} \| Ef_{Q, trans} \|_{L^5(\mathcal{N} \bigcap Q)}^5 \leq \sum_Q \| Ef_{Q, trans} \|_{L^5(Q)}^5$.
We can cover any $Q$ by a ball of radius $\frac{R}{G}$.
By Proposition \ref{lemma:wpd:cs} we can choose a wave packet composition of size $R/G$, and 
use the induction hypothesis for the smaller ball.
Then we have
\begin{equation}
\| Ef_{Q, trans} \|_{L^5(Q)}^5 \leq
C\left(\frac{1}{20}, \frac{1}{5}, 5, \epsilon \right)^5 \left(\frac{R}{G}\right)^{\frac{1}{4} + 5 \epsilon } 
\left(M G^{\frac{1}{4}}\right)^{1 - 5\epsilon} \| f_{Q, trans} \|_{L^2}^{4 + 5\epsilon}.
\end{equation}

For each $\theta, v$, $|\Omega_{\theta, v, trans}| \lesssim D$,
and each $\mathcal{N}_i \in \Omega_{\theta, v, trans}$ can be covered by $\sim 1$ squares.
Then by Proposition \ref{prop:wpd}
there is
\begin{equation}
\sum_{Q} \| f_{Q, trans} \|_{L^2}^2
\lesssim \sum_{Q} \sum_{\substack{\theta, v  \\ \exists \mathcal{N}_i \in \Omega_{\theta, v, trans}, \mathcal{N}_i \bigcap Q \neq \emptyset}} | c_{\theta, v} |^2   
\lesssim D \sum_{\theta, v} | c_{\theta, v} |^2
\sim D \left\| f \right\|_{L^2}^2  ,
\end{equation}
then $\sum_Q \| f_{Q, trans} \|_{L^2}^{4 + 5\epsilon} \lesssim D^{2 + \frac{5}{2}\epsilon} \left\| f \right\|_{L^2}^{4 + 5 \epsilon}$.

By taking $G$ large enough (relying on $\epsilon, \delta, D$),
we have
\begin{equation}   \label{eq:pp:pf5}
\| Ef_{Q, trans} \|_{L^5(B_R)}^5
=
\sum_Q \| Ef_{Q, trans} \|_{L^5(Q)}^5
\leq
\frac{1}{8}
C\left(\frac{1}{20}, \frac{1}{5}, 5, \epsilon \right)^5 R^{\frac{1}{4} + 5 \epsilon } 
M^{1 - 5\epsilon} \| f \|_{L^2}^{4 + 5\epsilon}.
\end{equation}

Now consider 
$\sum_Q \| Ef_{Q, non} \|_{L^5(\mathcal{N}\bigcap Q)}^5$.
Define $g_{tang}, g_{o} : B_R \rightarrow \C$ as following:
for any $Q \in \mathbb{S}$ and point $q \in Q$, let
\begin{equation}
g_{tang}(q) = \sum_{\substack{\theta, v : \\ \forall \mathcal{N}_i \in \Omega_{\theta, v, trans},  \mathcal{N}_i \bigcap Q = \emptyset \\ \exists \mathcal{N}_j \in \Omega_{\theta, v, tang},  q \in \mathcal{N}_j  }} Ef_{\theta, v}(q)   , \quad
g_{o}(q) = \sum_{\substack{\theta, v : \\ \forall \mathcal{N}_i \in \Omega_{\theta, v, trans}, \mathcal{N}_i \bigcap Q = \emptyset \\ \forall \mathcal{N}_j \in \Omega_{\theta, v, tang},  q \not\in \mathcal{N}_j  }} Ef_{\theta, v}(q)   .
\end{equation}
Then for any $q \in Q$, $Ef_{Q, non}(q) = g_{tang}(q) + g_{o}(q)$, and we have
\begin{equation}
\sum_Q \| Ef_{Q, non} \|_{L^5(\mathcal{N}\bigcap Q)}^5 \lesssim \| g_{tang} \|_{L^5(\mathcal{N})}^5 + \| g_{o} \|_{L^5(\mathcal{N})}^5.
\end{equation}
For any $q \in \mathcal{N}$, if $q$ is not in any $\mathcal{N}_i \in \Omega_{\theta, v, trans} \bigcup \Omega_{\theta, v, tang}$, from (\ref{eq:pp:pf2}), $q \not\in R^{\frac{1}+\delta} T_{\theta, v}$.
From Corollary \ref{cor:wpd:bd}, this implies that
$
|g_o (q)|^5 \lesssim_{\delta} R^{-100} \| f \|_{L^2}^5 
$
and 
\begin{equation} \label{eq:pp:pf6}
\| g_{o} \|_{L^5(\mathcal{N})}^5 \lesssim_{\delta}
R^{-50} \| f \|_{L^2}^5 .
\end{equation}

For $\| g_{tang} \|_{L^5(\mathcal{N})}^5$, we show that for each point $q \in \mathcal{N}$,
it is contained in $\lesssim D^2 R^{2\delta}$ sets in $\bigcup_{\theta, v} \Omega_{\theta, v, tang}$.

Suppose that there are $W$ sets in $\bigcup_{\theta, v} \Omega_{\theta, v, tang}$ that contains $q$.
For any $\theta, v$ and $\mathcal{N}_i \in \Omega_{\theta, v, tang}$, $q \in \mathcal{N}_i$,
there are two points $a, b \in \mathcal{N}_i$ with $|a - b| \geq R/G$.
Then there is $c \in \{a, b\}$ such that $|c - q| \geq R/2G$.
This implies that there are points $c', q' \in Z(\mathcal{F}_D)$, such that the segment of $Z(\mathcal{F}_D)$ between $c', q'$ is in $2R^{\frac{1}{2} + \delta}T_{\theta, v}$, and $|c - c'|, |q - q'| \leq R^{\frac{1}{2} + \delta}$.
Let $P_{\theta, v, i}$ be the segment of $Z(\mathcal{F}_D)$ between $c'$ and $q'$.

Denote $Y_1$ and $Y_2$ to be the circles centered at $q$ with radius $R/3G$ and $R^{\frac{1}{2} + \delta}$.
Then each $P_{\theta, v, i}$ intersects both $Y_1$ and $Y_2$.
As $Z(\mathcal{F}_D)$ intersects either $Y_1$ and $Y_2$ at $\lesssim D$ points, there are 
$\gtrsim W/D^2$ different $P_{\theta, v, i}$, that intersect $Y_1$ and $Y_2$ at the same point, respectively.
Let the two points be $y_1 \in Y_1$ and $y_2 \in Y_2$, and the segments be
$P_{\theta_1, v_1, i_1}, \cdots, P_{\theta_k, v_k, i_k}$, for $k \gtrsim W/D^2$.
Since for any fixed $\theta, v$ and $i \neq j$, $P_{\theta, v, i}$ does not overlap with $P_{\theta, v, j}$,
the pairs $(\theta_1, v_1), \cdots, (\theta_k, v_k)$ are mutually different.
Now there are $k$ tubes $R^{\frac{1}{2} + \delta}T_{\theta_1, v_1}, \cdots, R^{\frac{1}{2} + \delta}T_{\theta_k, v_k}$, each containing $y_1, y_2$.
Note that as $|y_1 - y_2| \sim R/G$, for any $R^{\frac{1}{2} + \delta}T_{\theta, v}$ containing both of them, $\theta$ can take only $GR^{\delta}$ values, and $v$ can take only $R^{\delta}$ values.
Then we have that $W/D^2 \lesssim k \leq GR^{2\delta}$, and $W \lesssim D^2 GR^{2\delta}$.

By Theorem \ref{thm:pm:wv}, we conclude that
\begin{equation}   \label{eq:pp:pf9}
\| g_{tang} \|_{L^5(\mathcal{N})}^5 = \int_{\mathcal{N}} |g_{tang}|^5
\lesssim (D^2 GR^{2\delta})^4 \int_{\mathcal{N}} \sum_{\theta, v}  R^{-\frac{5}{4}}|c_{\theta, v}|^5
\lesssim
(D^2 G)^4  R^{\frac{1}{4} + 9\delta} M \| f \|_{L^2}^4 .
\end{equation}
Finally, add (\ref{eq:pp:pf6}) and (\ref{eq:pp:pf9}) together, and 
take $C\left(\frac{1}{20}, \frac{1}{5}, 5, \epsilon\right)$,
relying on $G, D, \epsilon$,
such that when $R$ large enough. 
Then there is
\begin{equation}   \label{eq:pp:pf15}
\sum_{Q\in \mathbb{S}}\| Ef_{Q, non} \|_{L^5(\mathcal{N}\bigcap Q)}^5
\leq
\frac{1}{8}
C\left(\frac{1}{20}, \frac{1}{5}, 5, \epsilon \right)^5 R^{\frac{1}{4} + 5 \epsilon } 
M^{1 - 5\epsilon} \| f \|_{L^2}^{4 + 5\epsilon} ,
\end{equation}
and
\begin{multline}
\| Ef \|_{L^5(B_R)}^5 \leq 4\| Ef \|_{L^p(\mathcal{N})}^5 \leq 
4\sum_{Q\in \mathbb{S}}\| Ef_{Q, trans} \|_{L^5(\mathcal{N}\bigcap Q)}^5
+ \| Ef_{Q, non} \|_{L^5(\mathcal{N}\bigcap Q)}^5
\\
\leq
C\left(\frac{1}{20}, \frac{1}{5}, 5, \epsilon \right)^5 R^{\frac{1}{4} + 5 \epsilon } 
M^{1 - 5\epsilon} \| f \|_{L^2}^{4 + 5\epsilon} ,
\end{multline}
and taking $\frac{1}{5}$th power in both sides finishes the proof.
\end{case}
\end{proof}

\bigskip
\section{Using $l^2$ Decoupling Theorem: Proof for $p=6$}  \label{sec:dc}

Now we directly apply the $l^2$ Decoupling Theorem to give another bound for $p = 6$.
The theorem was proved in recent years by Bourgain and Demeter for compact surfaces with positive definite second fundamental form.
Some more explanations on the ideas in the proof can be found at \cite{bourgain2016study}.

In $1+1$ dimension and parabola, it can be stated as follows.
\begin{thm} \cite[Theorem 1.1]{MR3374964} \label{thm:dec}
For $\delta > 0$,
let $\mathbb{S}$ be a set of segments of $P$ such that
finitely covers $P$,
and each segment in $\mathbb{S}$ is of length $\sim \delta^{\frac{1}{2}}$.
For any smooth function $g : \R^2\rightarrow \C$,
if $\widehat{g}$ is supported in the $\delta$-neighborhood of $P$,
and we write
\begin{equation}
g = \sum_{\iota \in \mathbb{S}} g_{\iota} ,
\end{equation}
where each $\widehat{g}_{\iota}$ is supported in the
$\delta$-neighborhood of $\iota$ ;
then for any $\epsilon > 0$, there is
\begin{equation}
\| g \|_{L^6} \lesssim_{\epsilon} \delta^{-\epsilon} \left( \sum_{\iota \in \mathbb{S}}
\| g_{\iota} \|_{L^6}^2 \right)^{\frac{1}{2}}  .
\end{equation}
\end{thm}

We directly use Theorem \ref{thm:dec}
to obtain another sharp bound for $p=6$.

\begin{prop}  \label{prop:dec}
For $p = 6$,
(\ref{eq:mainieq}) holds for $\alpha = \frac{1}{6}$ and $\beta = \frac{2}{3}$  .
\end{prop}

\begin{proof}
Consider $Ef \cdot \eta_R = \sum_{\theta} Ef_{\theta} \cdot \eta_R$: 
each $\widehat{Ef_{\theta} \cdot \eta_R} = \widehat{Ef_{\theta}} \ast \widehat{\eta_R} = \widehat{f}_{\theta} \ast \widehat{\eta_R}$, and is supported in the $O(R^{-1})$ neighborhood of a segment of $P$, where the segment is the $\sim R^{-\frac{1}{2}}$ neighborhood of $(\theta, \theta^2)$ on $P$.
Also $\widehat{Ef\cdot \eta_R} = \widehat{Ef_{\theta}} \ast \widehat{\eta_R}$  is supported in the $O(R^{-1})$ neighborhood of $P$.
By Theorem \ref{thm:dec} we conclude that 
\begin{equation}  
\| Ef \|_{L^6(B_R)}
\lesssim \| Ef \cdot \eta_R \|_{L^6}
\lesssim_{\epsilon} R^{\epsilon} \left( \sum_{\theta}
\| Ef_{\theta} \cdot \eta_R \|_{L^6}^2 \right)^{\frac{1}{2}}  .
\end{equation}
For each $\theta$, there is
\begin{multline}
\| Ef_{\theta} \cdot \eta_R \|_{L^6}^6 = \int \left| Ef_{\theta} \cdot \eta_R \right|^6 
\leq \sum_{u \in R^{\frac{1}{2}}\Z } \int_{RT_{\theta, u}} \left| Ef_{\theta} \cdot \eta_R \right|^6 
\\
\lesssim \sum_{u,v \in R^{\frac{1}{2}}\Z} \int_{RT_{\theta, u}} \left| Ef_{\theta, v} \cdot \eta_R \right|^6 (1 + |u-v| R^{-\frac{1}{2}} )^6 .
\end{multline}
In (\ref{eq:wpd:bd}) we take $N = 100$, and get
\begin{multline}  \label{eq:pl2:pf1}
\| Ef_{\theta} \cdot \eta_R \|_{L^6}^6
\lesssim
\sum_{u,v \in R^{\frac{1}{2}}\Z} R^{-\frac{3}{2}}(1 + |u-v| R^{-\frac{1}{2}} )^6 \int_{RT_{\theta, u}} \left| \left(1 + |x - v - \theta t| R^{-\frac{1}{2}}\right)^{-100} \eta_R c_{\theta, v} \right|^6 
\\
\lesssim
\sum_{u,v \in R^{\frac{1}{2}}\Z} R^{-\frac{3}{2}}(1 + |u-v| R^{-\frac{1}{2}} )^{-2} \int_{RT_{\theta, u}} \left| \eta_R c_{\theta, v} \right|^6 .
\end{multline}
Use (\ref{eq:sf:bd}), taking $N = 10$, and (\ref{eq:pl2:pf1}) is bounded by
\begin{equation}
\sum_{u,v \in R^{\frac{1}{2}}\Z} R^{-\frac{3}{2}}(1 + |u-v| R^{-\frac{1}{2}} )^{-2} \int_{RT_{\theta, u}} \left| c_{\theta, v} \right|^6 (1 + |t|/ R)^{-10}
\lesssim
\sum_{u,v \in R^{\frac{1}{2}}\Z} (1 + |u-v| R^{-\frac{1}{2}} )^{-2} \left| c_{\theta, v} \right|^6 ,
\end{equation}
and is bounded by $\sum_{v \in R^{\frac{1}{2}}\Z} \left| c_{\theta, v} \right|^6 \leq M^4\sum_{v \in R^{\frac{1}{2}}\Z} \left| c_{\theta, v} \right|^2$, by summing over $u$.

Finally, there is
\begin{equation}
\| Ef \|_{L^6(B_R)} \lesssim_{\epsilon} R^{\epsilon } M^{\frac{2}{3}}
\left( \sum_{\theta} \left( \sum_{v} |c_{\theta, v}|^2 \right)^{\frac{1}{3}} \right)^{\frac{1}{2}}
\lesssim
R^{\epsilon } M^{\frac{2}{3}}
\left( R^{\frac{1}{3}} \left( \sum_{\theta, v} |c_{\theta, v}|^2 \right)^{\frac{1}{3}} \right)^{\frac{1}{2}},
\end{equation}
and it is $
\sim R^{\frac{1}{6} + \epsilon} S^{\frac{2}{3}}  \| f \|_{L^2}$.
\end{proof}

\bigskip
\section{Examples}  \label{sec:em}

In this section we give examples, which provide necessary conditions for (\ref{eq:mainieq}) to hold.

The example of Many Wave Packets is constructed to make $S$ arbitrarily small, showing that there must be $\beta \leq 1$ for (\ref{eq:mainieq}) to hold.
The next two scenarios are constructed to maximize the left hand side of (\ref{eq:mainieq}), by overlapping wave packets.
In the Bundle Scenario, we take a large area inside $B_R$ with equal number of overlaps inside the area;
in Figure \ref{fig:1} this corresponds to the two planes intersect at the line $XF$.
In the Star Scenario, we make every non-zero wave packet pass through the origin, thus make large number of overlaps in a small area;
in Figure \ref{fig:1} this corresponds to the two planes intersect at the line $YF$.
It is not surprising that the setting of the Star Scenario is larger for large $p$, since there $|Ef|$ highly concentrates in a small area and is small elsewhere;
while the Bundle Scenario is larger for small $p$, since there $|Ef|$ is ``equally large'' in a nontrivial proportion of $B_R$.

\subsection{Many Wave Packets}
We start with a trivial example. 
Take any $U \in \left(0, R^{-\frac{1}{2}} \right)$,
let $f: [-1, 1] \rightarrow \C$ be a smooth function, satisfying $f(\omega) = 1$ for $|\omega| < \frac{U}{2}$, $|f| < 1$, and $f$ is supported in $[- U, U]$.
\begin{prop}
To make (\ref{eq:mainieq}) hold, there must be $\beta \leq 1$.
\end{prop}

\begin{proof}
Under the wave packet decomposition (of size $R$) in Proposition \ref{prop:wpd}, there is $f_0 = f$ and any $f_{\theta} = 0$, for $\theta \neq 0$.
By direct computation or Theorem \ref{thm:sp}, there is $|Ef_0(x, 0)| \lesssim U(1 + |x|U)^{-N}$ for any $N > 0$, and $|Ef_0(x, t)| \sim U$ for $|x| \lesssim U^{-1}$ and $|t| \lesssim 1$.
Then each $|c_{0, v}| \lesssim R^{\frac{1}{4}}U $, 
$M \lesssim R^{\frac{1}{4}}U$
and $\| Ef_0 \|_{L^p(B_R)} \gtrsim R^{\frac{1}{p}} U$.
Also $\| f \|_{L^2} \sim U^{\frac{1}{2}}$.

Then (\ref{eq:mainieq}) becomes
\begin{equation}
R^{\frac{1}{p}} U \lesssim_{\epsilon} R^{\alpha + \epsilon} \left( R^{\frac{1}{4}}U \right)^{\beta - \epsilon} U^{\frac{1 - \beta + \epsilon}{2}} .
\end{equation}
Since $U$ can be taken arbitrarily small, there must be $1 \geq \beta - \epsilon + \frac{1 - \beta + \epsilon}{2} = \frac{1}{2} + \frac{\beta}{2} - \frac{\epsilon}{2}$, for any $\epsilon > 0$.
This implies that $\beta \leq 1$.
\end{proof}


\subsection{Bundle Scenario}  \label{sec:ex:bs}

Then consider the configuration where the wave packets take consecutive $\sim N$ directions (or $\theta$ in the decomposition),
and in every direction there are $\sim N$ locations (or $v$ in the decomposition)
centered at the origin.
More specifically,
for any $R > 0$, and positive integer $N < R^{\frac{1}{2}}$,
consider the following function
\begin{equation}
f_{R, N} : [-1, 1]\rightarrow \C , \quad f_{R, N}(\omega) = \sum_{n = -N}^N \Theta_{R, N}(\omega - n R^{-\frac{1}{2}})   ,
\end{equation}
where $\Theta_{R, N}: [-1, 1] \rightarrow \R$ is the smooth function,
such that $|\Theta_{R, N}| \leq 1$, $\Theta_{R, N}(\omega) = 1$ for $\omega \in \left[-\frac{1}{2}R^{-\frac{1}{2}}N^{-1}, \frac{1}{2}R^{-\frac{1}{2}}N^{-1}\right]$, and $\Theta_{R, N}$ is supported in $\left[-R^{-\frac{1}{2}}N^{-1}, R^{-\frac{1}{2}}N^{-1}\right]$.
\begin{prop}
To make (\ref{eq:mainieq}) hold, there must be $4p\alpha + p \geq 6$ and $2p \alpha - p \beta + p \geq 4$.
\end{prop}
\begin{proof}
We can do a wave packet decomposition of $f_{R,N}$ in the following way.
For $\theta = nR^{-\frac{1}{2}}$,
$-N\leq n \leq N$, denote $f_{\theta}: [-1, 1] \rightarrow \C$, $f_{\theta}(\omega)= \Theta_{R, N}(\omega - \theta) $.
Then each $f_{\theta}$ is supported in an interval of length $2R^{- \frac{1}{2}}$
centered at $\theta$;
and we compute $Ef_{\theta}$ as
\begin{equation} \label{eq:bs:onecap}
\int \me^{\im (\omega^2 t + \omega x)} f_{\theta}  \diff \omega 
= \me^{\im \left( n^2 R^{-1} t + n R^{-\frac{1}{2}}x \right)}  \int \me^{\im \left( \omega^2 t + \omega \left(2 n R^{-\frac{1}{2}} t  + x \right) \right) } \Theta_{R, N}(\omega)  \diff \omega  .
\end{equation}
By Theorem \ref{thm:sp},
the absolute value of (\ref{eq:bs:onecap}) is
$\sim R^{-\frac{1}{2}}N^{-1} $, 
when $\left| 2 n R^{-\frac{1}{2}} t + x \right| < R^{\frac{1}{2}}N$ and $(x, t) \in B_R$,
and fast decays outside this area.
This means that we can split
each $f_{\theta}$ into $2N$ wave packets
of equal size,
thus there is
\begin{equation}  \label{eq:ssize}
S = \frac{\left(\frac{1}{2N} \int  |\Theta_{R, N}(\omega)|^2 \diff \omega \right)^{\frac{1}{2}} }{\|f_{R, N}\|_{L^2}}
\sim N^{-1} .
\end{equation}

The $L^2$-norm of $f_{R, N}$ can be easily computed as
\begin{equation}
\|f_{R, N}\|_{L^2} = \left(\int |f_{R,N}(\omega)|^2 d\omega \right)^{\frac{1}{2}} = \left( (2N + 1) 2R^{-\frac{1}{2}}N^{-1} \right)^{\frac{1}{2}}
\sim R^{-\frac{1}{4}}  .
\end{equation}

Now we show that
for $2 \leq p \leq 6$ and any $\epsilon > 0$, there is
\begin{equation}  \label{eq:eg:bs:pf1}
\|Ef_{R, N}\|_{L^p(B_R)} \gtrsim_{\epsilon, p} R^{\frac{3}{2p}-\frac{1}{2}-\epsilon} N^{\frac{1}{p} - \frac{1}{2}} .
\end{equation}

For $p = 2$, since Fourier transform is unitary, and the $L^2$ norm is conserved, there is
\begin{equation}
\|Ef_{R, N}\|_{L^2(B_R)}^2 
\leq \|Ef_{R, N}\|_{L^2(\R\times [-R, R])}^2
= 2R \|Ef_{R, N}|_{t=0} \|_{L^2_x}^2
= 2R \| f \|_{L^2}^2
\sim R^{\frac{1}{2}}  
.
\end{equation}

For $p=6$, 
as seen in the proof of Proposition \ref{prop:dec}, 
there is
\begin{equation}
\|Ef_{R, N}\|_{L^6(B_R)}
\lesssim_{\epsilon}  R^{\epsilon} \left( \sum_{\theta} \| Ef_{\theta} \cdot \eta_R \|_{L^6}^2 \right)^{\frac{1}{2}}  .
\end{equation}
For each $\theta$,
we have that
\begin{equation}
\| Ef_{\theta} \cdot \eta_R \|_{L^6}
\sim \left( R^{-3}N^{-6}\cdot NR^{\frac{3}{2}} \right)^{\frac{1}{6}} = R^{-\frac{1}{4}} N^{-\frac{5}{6}} ,
\end{equation}
and thus
\begin{equation}
\|Ef_{R, N}\|_{L^6(B_R)}
\lesssim_{\epsilon} R^{- \frac{1}{4} + \epsilon} N^{- \frac{1}{3}}  .
\end{equation}

For $p=4$,
let $\delta = \frac{1}{10}$.
Then $\eta_{R^{1-\delta}}$ fast decays outside $B_{R^{1-\delta}}$,
and we have
\begin{equation}
\left\|Ef_{R,N} \cdot \eta_{R^{1 - \delta}}^{\frac{1}{2}} \right\|_{L^4}^4 \lesssim
\left\|Ef_{R,N} \right\|_{L^4(B_R)}^4 + R^{-100} \| f_{R, N} \|_{L^2}^4.
\end{equation}
By Plancherel Theorem
we have that
\begin{multline}
\left\| Ef_{R,N} \cdot \eta_{R^{1 - \delta}}^{\frac{1}{2}}  \right\|_{L^4}^4 = \int  \left| \eta_{R^{1 - \delta}} \sum_{\theta} Ef_{\theta} \right|^4
= \sum_{\theta_1, \theta_2, \theta_3, \theta_4} \int |\eta_{R^{1-\delta}}|^2 Ef_{\theta_1}Ef_{\theta_2} \overline{ Ef_{\theta_3} Ef_{\theta_4} }
\\
= \sum_{\theta_1, \theta_2, \theta_3, \theta_4 } \int \left( \widehat{\eta}_{R^{1 - \delta}} \ast {f}_{\theta_1}^* \diff \sigma \ast {f}_{\theta_2}^* \diff \sigma \right) \overline{\left( \widehat{\eta}_{R^{1 - \delta}} \ast {f}_{\theta_3}^* \diff \sigma \ast {f}_{\theta_4}^* \diff \sigma \right)}  ,
\end{multline}
and for any $\theta_1, \theta_2$, the convolution
$\widehat{\eta}_{R^{1 - \delta}} \ast {f}_{\theta_1}^* \diff \sigma \ast {f}_{\theta_1}^* \diff \sigma$
is real and non-negative in the 
Minkowski sum of their support,
and vanishes outside.
Then for any $\theta_1, \theta_2, \theta_3, \theta_4 $, the integral
\begin{equation}
\int |\eta_{R^{1 - \delta}}|^2 Ef_{\theta_1}Ef_{\theta_2} \overline{ Ef_{\theta_3} Ef_{\theta_4} }
\end{equation}
is non-negative.
Then there is
\begin{multline}
\left\|\eta_{R^{1 - \delta}}^{\frac{1}{2}} Ef_{R,N}\right\|_{L^4}^4 \gtrsim 
\sum_{\theta_1, \theta_2} \int |\eta_{R^{1 - \delta}}|^2 |Ef_{\theta_1}|^2 |Ef_{\theta_2}|^2 
=   \int |\eta_{R^{1 - \delta}}|^2 \left( \sum_{\theta} |Ef_{\theta}|^2 \right)^2  \\
\gtrsim  \int_{\mathcal{A}} \left( \sum_{\theta} |Ef_{\theta}|^2 \right)^2  
\sim  R^{\frac{3}{2}}N \cdot R^{-2} N^{-2} = R^{-\frac{1}{2}} N^{-1}  ,
\end{multline}
where $\mathcal{A}$ is the area defined by $|x|\leq \frac{1}{2}R^{\frac{1}{2}}N$ and $|t| \leq \frac{R}{4}$,
where each $|Ef_{\theta}| \sim R^{-\frac{1}{2}}N^{-1}$, and $|\eta_{R^{1 - \delta}}| \sim 1$.
Thus when $R$ large enough,
there is
\begin{equation}
\|Ef_{R,N}\|_{L^4(B_R)}^4 \gtrsim R^{-\frac{1}{2}} N^{-1}  .
\end{equation}

Finally, by H\"{o}lder's inequality,
for any $2\leq p < 4$
there is
\begin{equation}
\|Ef_{R,N}\|_{L^p(B_R)} \geq \frac{\|Ef_{R,N}\|_{L^4(B_R)}^{\frac{12}{p} - 2} }{\|Ef_{R,N}\|_{L^6(B_R)}^{\frac{12}{p} - 3} }
\gtrsim_{\epsilon} R^{\frac{3}{2p}-\frac{1}{2}-\left( \frac{12}{p} - 3 \right)\epsilon} N^{\frac{1}{p} - \frac{1}{2}}  ,
\end{equation}
and for any $4 < p \leq 6$
there is
\begin{equation}
\|Ef_{R,N}\|_{L^p(B_R)} \geq \frac{\|Ef_{R,N}\|_{L^4(B_R)}^{2 - \frac{4}{p}} }{\|Ef_{R,N}\|_{L^2(B_R)}^{2 - \frac{4}{p}} }
\gtrsim R^{\frac{3}{2p}-\frac{1}{2}} N^{\frac{1}{p} - \frac{1}{2}}  .
\end{equation}
And these imply (\ref{eq:eg:bs:pf1}).

Finally, if (\ref{eq:mainieq}) is true for some $p, \alpha, \beta$, there is
\begin{equation}
R^{\frac{3}{2p} - \frac{1}{2} - \epsilon } N^{\frac{1}{p} - \frac{1}{2}} \lesssim_{\epsilon} R^{\alpha + \epsilon} N^{- \beta + \epsilon} R^{-\frac{1}{4}} .
\end{equation}
By taking $N = 1$ and $N = R^{\frac{1}{2}}$, we get $4p \alpha + p \geq 6$ and $2 p \alpha - p \beta + p \geq 4$, respectively.
\end{proof}

\subsection{Star Scenario}  \label{sec:ex:sc}

Let us consider another configuration,
where for consecutive $N$ directions,
there is exactly one wave packet at each direction, passing through the origin.

Namely,
for any $R > 0$, and positive integer $N < R^{\frac{1}{2}}$,
consider the following function
\begin{equation}
f_{R, N} : [-1, 1]\rightarrow \C , \quad f_{R, N}(\omega) = \sum_{n = -N}^N  \Phi_{R, N}(\omega - n R^{-\frac{1}{2}})   ,
\end{equation}
where $\Phi_{R, N}: [-1, 1] \rightarrow \R$ is the smooth function,
such that $|\Phi_{R, N}| \leq 1$, $\Phi_{R, N}(\omega) = 1$ for $\omega \in \left[-\frac{1}{2}R^{-\frac{1}{2}}, \frac{1}{2}R^{-\frac{1}{2}}\right]$, and $\Phi_{R, N}$ is supported in $\left[- R^{-\frac{1}{2}}, R^{-\frac{1}{2}} \right]$.

We also require that $\Phi_{R, N}$ is an even function, and for any $\omega \in \left[0, \frac{1}{2}R^{-\frac{1}{2}}\right]$, there is $\Phi_{R, N}(\omega) = \Phi_{R, N}\left( \frac{1}{2}R^{-\frac{1}{2}} - \omega \right)$.
Then $f_{R, N}(\omega) = 1$ for any $\omega \in \left[-\left(N + \frac{1}{2} \right)R^{-\frac{1}{2}}, \left(N + \frac{1}{2}\right)R^{-\frac{1}{2}}\right]$,
and $f_{R, N}$ is real and supported in
$\left[-\left(N + 1 \right)R^{-\frac{1}{2}}, \left(N + 1\right)R^{-\frac{1}{2}}\right]$.
\begin{prop}
To make (\ref{eq:mainieq}) hold, there must be $4p\alpha + p \geq 6$ and $4 \alpha - \beta \geq 0$.
\end{prop}



\begin{proof}
The wave packet decomposition of $f_{R,N}$ is as following:
for $\theta = nR^{-\frac{1}{2}}$,
$-N\leq n \leq N$, denote $f_{\theta}: [-1, 1] \rightarrow \C$, $f_{\theta}(\omega)= \Phi_{R, N}(\omega - \theta) $.
Then each $f_{\theta}$ is supported in an interval of length $2R^{- \frac{1}{2}}$
centered at $\theta$.
We compute $Ef_{\theta}$ as
\begin{equation}   \label{eq:ex:st}
\int \me^{\im (\omega^2 t + \omega x)} f_{\theta}  \diff \omega 
=  \me^{\im \left( n^2 R^{-1} t + n R^{-\frac{1}{2}}x \right)}\int \me^{\im \left( \omega^2 t + \omega \left(2 n R^{-\frac{1}{2}} t  + x \right) \right) } \Phi_{R, N}(\omega)  \diff \omega  .
\end{equation}
The absolute value of (\ref{eq:ex:st}) $\sim R^{-\frac{1}{2}}$
inside the area defined by $\left| 2 n R^{-\frac{1}{2}} t + x \right| < 2R^{\frac{1}{2}}$
and $|t| < R$;
and by Theorem \ref{thm:sp},
it also fast decays outside.
This means that each $f_{\theta}$ corresponds to $\sim 1$ wave packets,
of equal size,
in wave packet decomposition.
Thus we have $S \sim N^{-\frac{1}{2}}$.

As for the $L^2$-norm of $f_{R, N}$, there is
\begin{equation}
\|f_{R, N}\|_{L^2} \sim \left( \int_{-\left(N + \frac{1}{2}\right)R^{-\frac{1}{2}}}^{\left(N + \frac{1}{2}\right)R^{-\frac{1}{2}}} 1 \diff \omega  \right)^{\frac{1}{2}}
\sim R^{-\frac{1}{4}} N^{\frac{1}{2}} .
\end{equation}

Then it remains to estimate $\|Ef_{R, N}\|_{L^p(B_R)}$
for $2 \leq p \leq 6$.
Let $\mathcal{C}$ be the area defined by $|t| \leq \frac{1}{10}RN^{-2}$,
and $|x| \leq \frac{1}{10}R^{\frac{1}{2}}N^{-1}$.
Then
for any $(x, t)\in \mathcal{C}$,
there is
\begin{equation}
|Ef_{R, N}(x, t)| = 
\left|
\int \me^{\im (\omega^2 t + \omega x)} f_{R, N}  \diff \omega 
\right|
\sim R^{\frac{1}{2}}N  ,
\end{equation}
since $f_{R, N}$ is real and supported in
$\left[-\left(N + 1 \right)R^{-\frac{1}{2}}, \left(N + 1\right)R^{-\frac{1}{2}}\right]$.
We conclude that
\begin{equation}
\|Ef_{R, N}\|_{L^p(B_R)} \geq
\|Ef_{R, N}\|_{L^p(\mathcal{C})} \gtrsim
\left( R^{\frac{3}{2}}N^{-3} \cdot N^p R^{\frac{p}{2}}  \right)^{\frac{1}{p}}
\geq R^{\frac{3}{2p} - \frac{1}{2}} N^{1 - \frac{3}{p}}   .
\end{equation}

To make (\ref{eq:mainieq}) hold, there must be
\begin{equation}
R^{\frac{3}{2p} - \frac{1}{2}} N^{1 - \frac{3}{p}}
\lesssim_{\epsilon}
R^{\alpha + \epsilon} N^{\frac{\epsilon - \beta}{2}} R^{-\frac{1}{4}} N^{\frac{1}{2}} ,
\end{equation}
and via taking $N = 1$ and $N=R^{\frac{1}{2}}$ we conclude $4p\alpha + p \geq 6$ and $4 \alpha - \beta \geq 0$, respectively.
\end{proof}

\bigskip
\section{Conjecture}  \label{sec:cj}

We make the conjecture that (\ref{eq:thm:main:rv}) is sufficient for 
 (\ref{eq:mainieq}) to be true.
It suffices to consider the follow case.
\begin{conj}
The inequality (\ref{eq:mainieq}) holds for $p = \frac{14}{3}$,
$\alpha = \frac{1}{14}$
and $\beta = \frac{2}{7}$. 
\end{conj}
This corresponds to point $F$ in Figure \ref{fig:1}.
Again, by H\"{o}lder's inequality and Lemma \ref{lemma:ext}, this implies that all possible $p, \alpha, \beta$ are given by (\ref{eq:thm:main:rv}), and the remaining cases can be settled.

\bibliographystyle{halpha}
\bibliography{bibliography}

\end{document}